\documentclass[hidelinks,onefignum,onetabnum]{siamart220329}

\usepackage{lipsum}
\usepackage{amsfonts}
\usepackage{graphicx}
\usepackage{epstopdf}
\usepackage{algorithmic}
\usepackage{calligra}
\usepackage{amsfonts,amsmath,amssymb}
\usepackage{mathtools}
\usepackage{hyperref}
\usepackage{autonum}
\usepackage{hhline}
\usepackage{array}
\usepackage{diagbox}
\usepackage{tcolorbox}
\usepackage{mdframed}
\usepackage{multicol}
\usepackage{graphicx}
\usepackage{subcaption}
\usepackage{moreverb}
\usepackage{bbm}
\usepackage{todonotes}
\usepackage{scalerel,amssymb}
\allowdisplaybreaks
\usepackage{mathrsfs}  
\usepackage{lineno}
\usepackage{todonotes}
\usepackage{tikz}
\usetikzlibrary{external}
\usepackage{pgfplots}
\usetikzlibrary{arrows.meta}
\usepackage{siunitx}
\usepackage[numbers,sort&compress]{natbib}
\definecolor{mygreen}{HTML}{43a047}
\usepackage{subcaption}
\usepackage{doi}
\usepackage{alphalph}
\usepackage{booktabs}
\usepackage{makecell}
\definecolor{darkgreen}{rgb}{0,0.5,0}

\usepackage{enumitem}
\usepgfplotslibrary{groupplots}
\usepackage{dutchcal}
\usepackage[notocbasic]{nomencl}
\usepackage{etoolbox}
\usepackage{cancel}
\usepackage[normalem]{ulem}
\usepackage{appendix}
\usepackage{wrapfig}
\def\CLinf{C(L^\infty(\Omega))}

\def\ConeLinf{C^1(L^\infty(\Omega))}
\def\CtwoLinf{C^2(L^\infty(\Omega))}
\def\calA{\mathcal{A}}
\def\calS{\mathcal{S}}

\def\calT{\mathcal{T}}

\def\eps{\varepsilon}

\def\ulR{\underline{R}}
\def\olR{\overline{R}}
\def\op{\overline{p}}

\newcommand{\Om}{\Omega}

\def\Xp{{X}_p}
\def\frakXp{\mathfrak{X}_p}

\def\BR{\mathcal{B}_R}
\def\frakBR{\mathfrak{B}_R}
\def\tR{{\overline{\mathcal{R}}}}
\def\tRs{{\overline{\mathcal{R}}}^*}

\def\tdelta{\tilde{\delta}}
\def\tdeltap{\tilde{\delta}_p}

\def \Dtalpha{\textup{D}_t^\alpha}
\newcommand{\bmass}{\mathcal{m}}
\newcommand{\tspace}{\hspace{0.03cm}}

\def\pone{p^{(1)}}
\def\ptwo{p^{(2)}}
\def\fone{f_1}
\def\ftwo{f_2}
\def\frakS{\mathfrak{S}}
\def\frakXp{\mathfrak{X}_p}
\def\frakT{\mathfrak{T}}
\newcommand{\Rt}{R_t}
\newcommand{\Rtt}{R_{tt}}

\newcommand{\pstat}{p_{\textup{stat}}}
\newcommand{\pint}{p_{\textup{int}}}
\newcommand{\pext}{p_{\textup{ext}}}
\newcommand{\pv}{p_{\textup{v}}}
\newcommand{\pb}{p_{\textup{b}}}
\newcommand{\ppgn}{p_{\textup{pgn}}}
\newcommand{\Rn}{R_0}
\newcommand{\Req}{R_0}
\newcommand{\Rs}{R^*}
\newcommand{\Rsone}{R^{(1),*}}
\newcommand{\Rstwo}{R^{(2),*}}
\newcommand{\Rst}{R^*_t}
\newcommand{\Rstt}{R^*_{tt}}

\newcommand{\pt}{p_t}

\newcommand{\Dt}{\textup{D}_t}

\def\ts{\tilde{s}}
\def\dtsds{\, \textup{d}\tilde{s}\textup{d}s}
\def\hzero{h_0}
\def\rhoinv{\tfrac{1}{\rho_0}}
\newcommand{\ds}{\, \textup{d} s }

\newcommand{\dxs}{\, \textup{d}x\textup{d}s}

\newcommand{\inttO}{\int_0^t \int_{\Omega}}

\newcommand{\intO}{\int_{\Omega}}
\def\intt{\int_0^t}

\def\bfu{\boldsymbol{u}}
\newcommand{\deltap}{\delta_p}
\newcommand{\deltaR}{\delta_R}
\newcommand{\R}{\mathbb{R}} 
\newcommand{\N}{\mathbb{N}} 
\newcommand{\Ltwo}{L^2(\Omega)}
\newcommand{\Linf}{L^\infty(\Omega)}
\newcommand{\Hone}{H^1(\Omega)}
\newcommand{\Htwo}{H^2(\Omega)}
\newcommand{\Hthree}{H^3(\Omega)}
\newcommand{\Honezero}{H_0^1(\Omega)}
\newcommand{\Honetwo}{{H_\diamondsuit^2(\Omega)}}
\newcommand{\Honethree}{{H_\diamondsuit^3(\Omega)}}
\newcommand{\LtwoTLtwo}{L^2(0,T; L^2(\Omega))}

\def\LtwoLtwo{L^2(\Ltwo)}

\def\LinfLinf{L^\infty(L^\infty(\Omega))}

\makenomenclature
\makeatletter
\patchcmd{\thenomenclature}{\section*}{\section*}{}{}
\makeatother

\definecolor{grey}{rgb}{0.5,0.5,0.5}
 
\newcommand{\vanja}[1]{{\textcolor{black}{ #1}}}

\usetikzlibrary{arrows.meta, quotes}% quotes for "syntax" on edges/tos
\makeatletter
\tikzset{
	node on line/.style={
		to path={
			\pgfextra{%
				\edef\tikz@temp{% rescuing nodes and target for edge
					edge[
					line to, path only, % line to = --, path only = no draw, no fill, …
					every edge quotes/.append style={auto=false},% node *on* the line
					nodes={alias=@nodeonline@}]
					coordinate(@nodeonline@)% fallback coordinate
					\unexpanded\expandafter{\tikz@tonodes}(\tikztotarget)
				}\expandafter
			}\tikz@temp
			-- (@nodeonline@) -- (\tikztotarget)}}}
\makeatother
\tikzset{global scale/.style={
		scale=#1,
		every node/.style={scale=#1}
	}
}

\usepackage{lipsum}
\usepackage{amsfonts}
\usepackage{graphicx}
\usepackage{epstopdf}
\usepackage{algorithmic}
\ifpdf
  \DeclareGraphicsExtensions{.eps,.pdf,.png,.jpg}
\else
  \DeclareGraphicsExtensions{.eps}
\fi

\newsiamremark{remark}{Remark}
\newsiamremark{hypothesis}{Hypothesis}
\crefname{hypothesis}{Hypothesis}{Hypotheses}
\newsiamthm{claim}{Claim}

\headers{Nonlinear ultrasound contrast imaging  \qquad \ \, }{Vanja Nikoli\'c and Teresa Rauscher}

\title{Mathematical models for nonlinear ultrasound contrast imaging with microbubbles}

\author{Vanja Nikoli\'c\thanks{Department of Mathematics,
		Radboud University,      
		Heyendaalseweg 135,    
		6525 AJ Nijmegen, The Netherlands (\href{vanja.nikolic@ru.nl}{vanja.nikolic@ru.nl}).}
\and Teresa Rauscher\thanks{Department of Mathematics, 
	University of Klagenfurt, Universit\"atsstra\ss e 65--67, A-9020 Klagenfurt, Austria (\href{teresa.rauscher@aau.at}{teresa.rauscher@aau.at}).}}

\usepackage{amsopn}

\usepackage{marginnote}

\begin{document}
	\maketitle  
	   
	\begin{abstract}
		Ultrasound contrast imaging is a specialized imaging technique that applies microbubble contrast agents to traditional medical sonography, providing real-time visualization of blood flow and vessels. Gas-filled microbubbles are injected into the body, where they undergo compression and rarefaction and interact nonlinearly with the ultrasound waves. Therefore, the propagation of sound through a bubbly liquid is a strongly nonlinear problem that can be modeled by a nonlinear acoustic wave equation for the propagation of the pressure waves coupled via the source terms to a nonlinear ordinary differential equation of Rayleigh--Plesset type  for the bubble dynamics. In this work, we first derive a hierarchy of such coupled models based on constitutive laws. We then focus on the coupling of Westervelt's acoustic equation to Rayleigh--Plesset type equations, where we rigorously show the existence of solutions locally in time under suitable conditions on the initial pressure-microbubble data and final time. Thirdly, we devise and discuss numerical experiments on both single-bubble dynamics and the interaction of microbubbles with ultrasound waves. 
	\end{abstract}


% REQUIRED
\section{Introduction}
		Ultrasound imaging is a crucial tool in medical diagnostics due to its non-invasive nature, real-time capability, and versatility. However, it has limitations in terms of precision. To address this issue, contrast agents based on microbubbles were introduced. These agents consist of microbubbles with a gas core encapsulated by a shell that resonates and produces echoes when exposed to ultrasound waves. This technique, known as \emph{ultrasound contrast imaging}, shows strong nonlinear effects due to the high-frequency pressure waves in the tissue, the interaction of sound waves with the bubbles, and the behavior of the microbubbles themselves. The bubbles can form, expand, or collapse -- a process known as acoustic cavitation. Acoustic cavitation can be non-inertial, involving stable periodic oscillations, or inertial, involving rapid growth and violent collapse of the bubbles. To minimize potential tissue damage, inertial cavitation is typically avoided.  Properly managing pressure in the focal region is critical to optimizing the imaging process, that is, ensuring clear images while maintaining safety and effectiveness in ultrasound contrast imaging technologies. In particular, monitoring the nonlinear effects can help maintain the balance between achieving high-quality imaging and minimizing unwanted cavitation effects.\\
	\indent With this motivation in mind, in this work we investigate the complex nonlinear problem of ultrasound contrast imaging from a mathematical side, focusing on non-inertial cavitation. \vanja{The problem of modeling and simulation of ultrasound contrast agents has recently gained attention; see, for example,~\cite{matalliotakis2023computation, matalliotakis2024impact, blanken2024proteus} and the references contained therein.}  Our main contributions in this regard are threefold. They pertain to (i) the derivation of the coupled nonlinear models of ultrasound-microbubble interactions, (ii) their rigorous mathematical analysis, and (iii) their numerical \vanja{study}. \\
	%simulation. \\
	\indent The models, derived in Section~\ref{Sec:Modeling} starting from constitutive laws, are based on a volume coupling of a PDE for the propagation of ultrasound waves with an ODE for the microbubble dynamics. Among them is the system based on the damped Westervelt's equation for the acoustic pressure $p=p(x,t)$ in bubbly media. This wave equation is given by
	\begin{equation} \label{West_eq}
		\begin{aligned}
			((1+ 2k(x)p) p_t)_t - c^2 \Delta p -b \Delta \pt=  c^2\eta v_{tt},
		\end{aligned}
	\end{equation}
	where $v= \frac43 \pi R^3$ in the source term denotes the volume of the microbubble with radius $R$; see Section~\ref{Sec:Modeling} for details. Westervelt's equation 
	is capable of capturing prominent nonlinear effects in acoustic wave propagation, such as the steepening of the wave front and generation of higher harmonics. In \eqref{West_eq}, $c>0$ is the speed of sound in the medium, $b$ the sound diffusivity, and the coefficient $\eta$ is related to the number of microbubbles per unit volume. The nonlinearity coefficient  $k=k(x)$  is allowed to vary in space, as relevant for imaging applications such as \emph{acoustic nonlinearity tomography}~\cite{kaltenbacher_rundell2021, kaltenbacher2022inverse}. \\ 
 \indent When coupling \eqref{West_eq} to an equation for the changes in the radius $R$ of the microbubbles, we have various established models to choose from, commonly referred to as Rayleigh--Plesset equations; see, e.g.,~\cite{lauterborn1999nonlinear, wolfrum2004cavitation, vokurka1986comparison, versluis2020ultrasound, doinikov2011review} and the references provided therein. These equations are derived under the assumptions that the bubble size is much smaller than the pressure wavelength, which, in turn, means we may assume that  microbubbles remain spherical as they oscillate. The equations describe the variation of the bubble radius $R$ and can be expressed in the unified form
	\begin{equation} \label{general_bubble_ODE}
		\begin{aligned}
			\rho_{l0}  \left[ R R_{tt} + \tfrac{3}{2} R_t^2 \right]	=\, \pint-\pext,
		\end{aligned}
	\end{equation}
	where $\pint$ is the internal and $\pext$ external pressure. The latter involves the acoustic input $p$, dictated by \eqref{West_eq}, and influences the growth and near collapse of the bubble. The coefficient $ \rho_{l0}$ in \eqref{general_bubble_ODE} is the mean mass density of the liquid. Depending on the effects they capture, different forms of $\pint=\pint(R, \Rt)$ and $\pext=\pext(R, \Rt, p)$ arise in the literature and lead to different versions of this ODE. We discuss them in detail in Section~\ref{Sec:Modeling}. In the analysis, we consider an equation for the dynamics of microbubbles with a generalized right-hand side that covers these various cases of practical interest; see Section~\ref{Sec:Analysis}. In the simulations in Section \ref{Sec:Numerics}, we focus on two concrete versions of Rayleigh--Plesset equations, where we examine both single-bubble dynamics under a sinusoidal driving pressure and the ultrasound-bubble interaction.
	\subsection*{Related literature and novelty} 
	To the best of our knowledge, \vanja{this work contains} the first \vanja{rigorous well-posedness study} of models of ultrasound-microbubble interactions. In general, mathematical literature on wave-ODE systems is quite limited. For our theoretical purposes, some helpful ideas can be sourced from other studies of PDE-ODE systems involving the Rayleigh--Plesset equations, in particular~\cite{biro2000analysis} and~\cite{vodak2018mathematical}. \\
	\indent The mathematical literature on nonlinear acoustic equations for non-bubbly media is, on the other hand, by now quite rich, especially when it comes to the Westervelt equation. In the presence of strong damping (that is, when $b>0$), this equation has a parabolic-like character and  its global behavior has been established rigorously in different settings in terms of data; we refer to~\cite{kaltenbacher2009global, meyer2011optimal} for the analysis in the presence of homogeneous Dirichlet conditions.  As the damping coefficient $b$ is relatively small in practice, the limiting behavior of nonlinear acoustic models as this parameter vanishes has also been investigated.  A $b$-uniform well-posedness and the related convergence analysis can be found  in~\cite{kaltenbacher2022parabolic}. Refined versions of nonlinear acoustic models for non-bubbly media, such as the Kuznetsov equation, have also been thoroughly investigated in the mathematical literature; see, for example,~\cite{mizohata1993global, dekkers2017cauchy}. For a detailed overview of mathematical research in nonlinear acoustics, we refer to the review paper~\cite{kaltenbacherMathematicsNonlinearAcoustics2015}. 
	
	\section{Derivation of nonlinear acoustic models in bubbly media} \label{Sec:Modeling}
	
	In this section, we derive a hierarchy of second-order wave equations for the propagation of sound through a bubbly media starting from a governing system of equations and then couple them to models of microbubble dynamics. To this end, we introduce the following quantities, which are decomposed into their mean and alternating part denoted by $\cdot_0$ and $\cdot'$, respectively:
	\begin{equation}\label{acoustic field}
		\begin{aligned}
			\text{density of the mixture:} \quad  & \rho = \rho_0 + \rho', &
			\text{density of the liquid:} \quad  & \rho_l = \rho_{l0} + \rho_l' , \\
			\text{pressure of the liquid:} \quad  & p = p_0 + p',  &
			\text{gas volume fraction:} \quad  &\bmass =\bmass_0 +\bmass' ,\\
			\text{bubble number density:} \quad  & n = n_0+n', &
			\text{bubble radius:} \quad  & R=R_0 + R', \\
			\text{liquid and gas velocity:} \quad  & \bfu = \bfu'.
		\end{aligned}
	\end{equation}
	In our setting the bubbles are assumed small and densely distributed such that they determine the properties of a continuum mixture of density $\rho$. The volume concentration of microbubbles is assumed to be uniform. As $\bmass$ is the gas volume fraction, that is the fraction of unit volume of mixture occupied by microbubbles, we have 
	\begin{equation}\label{eq: bfmass}
		\bmass=n v \ \text{ with } \ v=\frac{4}{3}\pi R^3,
	\end{equation}
	where $n$ is the bubble number density, or more precisely the concentration of bubbles per unit volume, and $R$ the bubble radius. Then the continuum density is given by
	\begin{equation} \label{eq: continuum density}
		\rho = (1-\bmass) \rho_l +\bmass \rho_g,
	\end{equation}
	where $\rho_l$ and $\rho_g$ are the liquid and gas densities in a suspension of gas bubbles in liquid, respectively. We assume that the mean pressure $p_0$ is the same in bubbles and liquid and that all bubbles have the same equilibrium radius $R_0$. 
	\subsection{Approach in the derivation} In the upcoming derivation, we adopt the approach of Crighton described in~\cite{Crighton1991}, following Lighthill's scheme introduced in~\cite{lighthill1956viscosity} and elaborated by Blackstock in ~\cite{blackstockApproximateEquationsGoverning1963}; see also \cite[Ch.\ 5]{hamilton1998nonlinear}. According to this approach, we distinguish three categories of contributions:
	\begin{itemize}
		\item First order. First-order contributions are terms that are linear with respect to the fluctuating quantities $(\cdot)'$ and do not relate to dissipative effects.
		\item Second order. Second-order contributions are linear dissipative terms and terms that are quadratic with respect to fluctuating quantities.
		\item Higher order. All remaining terms are considered to be of higher order. 
	\end{itemize}
	The fundamental rule in Lighthill's approximation scheme is that one should keep first- and second-order terms while higher order contributions are neglected. A further approximation rule that we make use of below is the \emph{substitution corollary} that permits us to substitute each term with second- or higher-order terms with its first-order approximation. In addition, we assume in this section that the fluctuating quantities are zero at initial time.

	\subsection{Derivation of acoustic equations in bubbly liquids}
	
	The acoustic field is fully described by the conservation of mass for the mixture
	\begin{equation}\label{eq: conservation of mass}
		\rho_t + \nabla \cdot ( \rho \bfu) = 0,
	\end{equation}
	the conservation of momentum for the mixture
	\begin{equation} \label{eq: conservation of momentum}
		\rho \left( \bfu_t + \left( \bfu \cdot \nabla\right)  \bfu\right) + \nabla p = \mu \Delta \bfu + \left( \frac{\mu}{3} + \mu_b \right) \nabla \left( \nabla \cdot \bfu \right),
	\end{equation}
	where $\mu$ is the shear viscosity and $\mu_b$ the bulk viscosity in the medium, the conservation of bubbles 
	\begin{equation} \label{eq: conservation of bubbles}
		n_t + \nabla \cdot ( n \bfu ) = 0,
	\end{equation}
	and the state equation
	\begin{equation} \label{eq: state equation}
		\rho_l' = \frac{p'}{c^2} - \frac{1}{\rho_{l0} c^4 } \frac{B}{2A} p'^2- \frac{\gamma }{\rho_{l0} c^4 } \left( \frac{1}{c_v} - \frac{1}{c_p} \right) p'_t,
	\end{equation}
	where $c$ is the speed of sound, $\frac{B}{A}$ denotes the parameter of nonlinearity and $\gamma$ the adiabatic exponent. Further, $c_v$ and $c_p$ are the specific heat capacity at constant volume and constant pressure, respectively.\\
	\indent The gas contribution $\bmass \rho_g$ to $\rho$ is relatively small and can therefore be neglected (following, e.g.,~\cite[Sec.\ 4]{Crighton1991}), such that \eqref{eq: continuum density} reduces to
	\begin{equation} \label{eq: continuum density l}
		\rho = (1-\bmass) \rho_l, 
	\end{equation}
	and thus we have $\rho_0 = (1-\bmass_0) \rho_{l0}$ and $\rho'=(1-\bmass_0) \rho_l' - \frac{\bmass' \rho_0}{1-\bmass_0}.$
	The gas contribution of the pressure will be incorporated in the ODE for the bubble radius.
	Recalling that the gas volume fraction is $\bmass = nv$, its mean and alternating parts are given by $	\bmass_0= n_0 v_0$ and $\bmass'= n_0 v' + v_0 n' + n'v'$, respectively.

	\indent Equations \eqref{eq: bfmass}, \eqref{eq: conservation of mass}, \eqref{eq: conservation of momentum}, \eqref{eq: conservation of bubbles}, \eqref{eq: state equation}, and \eqref{eq: continuum density l} together with an ODE for the bubble radius are a set of seven equations for ($\rho',\, \rho_l',\, p',\, \bmass',\, \bfu,\, n',\, R'$). Now, the goal is to approximate this system of equations by one wave equation for the pressure $p'$ and equip it with a Rayleigh--Plesset equation for the radius $R$ of microbubbles to altogether arrive at a coupled system that describes the propagation of sound through a bubbly liquid. \\
	\indent  We start by noting that below we only use the linearized approximation of equation \eqref{eq: conservation of bubbles} for the conservation of microbubbles; see \eqref{eq: bubbles linear}. We next proceed by approximating other constitutive equations and then combining them.
	
	\subsubsection{Simplifying the conservation of momentum equation} Consider equation \eqref{eq: conservation of momentum}. Recalling the following vector identities on a convex domain:
	\begin{equation}
		\left( \bfu \cdot \nabla \right) \bfu = \tfrac{1}{2} \nabla (\bfu \cdot \bfu) - \bfu \times \nabla \times \bfu, \qquad \nabla \left( \nabla \cdot \bfu \right) = \Delta \bfu + \nabla \times \nabla \times \bfu,
	\end{equation}
	and substituting \eqref{eq: continuum density l} into \eqref{eq: conservation of momentum}, we obtain 
	\begin{equation}
		(1-\bmass) \rho_l \left( \bfu_t + \tfrac{1}{2} \nabla \left( \bfu \cdot \bfu\right) \right) + \nabla p = \left( \tfrac{4}{3} \mu + \mu_b \right) \Delta \bfu,
	\end{equation}
	where certain terms on the right-hand side are neglected as they decay exponentially away from boundaries and eventually become small compared to the corresponding first- and second-order terms, according to \cite{hamilton1998nonlinear}. 
	The decomposition of the quantities into their mean and fluctuating parts, and recalling that $\nabla p_0=0$, leads to
	\begin{equation} \label{eq: momentum2}
		\begin{aligned}
			\left( 1 -\bmass_0\right)  \rho_{l0} \bfu_t + \nabla p'  =& \begin{multlined}[t] \rho_{l0}\bmass' \bfu_t - \left( 1 -\bmass_0\right)  \rho_{l}' \bfu_t \\
				- \tfrac{(1-\bmass_0)}{2} \rho_{l0} \nabla \left( \bfu \cdot \bfu \right)+\left( \tfrac{4}{3} \mu + \mu_b \right) \Delta \bfu. \end{multlined}
		\end{aligned}
	\end{equation}
	Here the first-order terms appear on the left hand-side, second-order terms are collected on the right-hand side and higher-order terms have been neglected. \\
	\indent Next, we make use of the substitution corollary, which means that we are allowed to substitute any physical quantity in a second-order term by its linear approximation since the resulting errors will be of third order. More precisely, we use the following:
	\begin{align}
		\text{linear state equation:} \quad & \rho_l' = \tfrac{1}{c^2} p' \label{eq: state linear}\\
		\text{linear momentum equation:} \quad & \rho_0  \bfu_t + \nabla p'=0 \label{eq: momentum linear}\\
		\text{linear continuity equation:} \quad & \nabla \cdot \bfu=- \tfrac{1}{\rho_0} \rho'_t =  - \tfrac{1-\bmass_0}{\rho_0 c^2} p'_t + \tfrac{1}{1-\bmass_0} \bmass'_t \label{eg: continuity linear} \\
		\text{linear conservation of bubbles:} \quad & \nabla \cdot \bfu =- \tfrac{1}{n_0} 
		n'_t \label{eq: bubbles linear} %\\
		%	\text{relation: } \qquad & p' p'_t = \frac{1}{2} (p'^2)_t.
	\end{align}
	as well as the identity $p' p'_t = \frac{1}{2} (p'^2)_t$. The terms on the right hand-side of equation \eqref{eq: momentum2} can therefore be approximated as follows:
	\begin{align}
		\rho_{l0}\bmass' \bfu_t &\approx -\tfrac{\rho_{l0}}{\rho_0}\bmass' \nabla p' = - \tfrac{1}{1-\bmass_0}\bmass' \nabla p',\\
		- \left( 1 -\bmass_0\right)  \rho_{l}' \bfu_t & \approx - \tfrac{1-\bmass_0}{c^2} p' \bfu_t= \tfrac{1-\bmass_0}{\rho_0 c^2} p' \nabla p' = \tfrac{1-\bmass_0}{2 \rho_0 c^2} \nabla p'^2, \\
		\left( \tfrac{4}{3} \mu + \mu_b \right) \Delta \bfu & \approx	\left( \tfrac{4}{3} \mu + \mu_b \right) \left( - \tfrac{1}{n_0} \nabla n'_t \right),
	\end{align}
	and we can also rely on the fact that $	- \frac{1-\bmass_0}{2} \rho_{l0} \nabla \left(\bfu \cdot \bfu \right) =- \frac{\rho_0}{2} \nabla \left( \bfu \cdot \bfu \right)$.
	In this manner, from \eqref{eq: momentum2} we arrive at its approximate version given by
	\begin{equation}
		\begin{aligned} \label{eq: momentum3}
			\rho_{l0} \bfu_t + \nabla p'  = & \begin{multlined}[t]- \tfrac{1}{1-\bmass_0}\bmass' \nabla p' +\tfrac{1-\bmass_0}{2 \rho_0 c^2} \nabla p'^2  
				- \tfrac{\rho_{l0}}{2} \nabla \left(\bfu \cdot \bfu \right) -	\tfrac{ \tfrac{4}{3} \mu + \mu_b }{n_0} \nabla n'_t. \end{multlined}
		\end{aligned}
	\end{equation}
	
	\subsubsection{Simplifying the conservation of mass equation} Now, let us apply a similar procedure to \eqref{eq: conservation of mass}. We substitute \eqref{eq: continuum density l} into \eqref{eq: conservation of mass} to obtain
	\begin{equation}
		\left( (1-\bmass)\rho_l \right)_t + \nabla \cdot \left( (1-\bmass) \rho_l \bfu \right)=0.
	\end{equation}
	The decomposition for the quantities into their mean and fluctuating parts yields
	\begin{equation}
		\begin{aligned}
			- \rho_{l0} \bmass'_t + (1-\bmass_0) \rho_{l \tspace t}' + (1-\bmass_0) \rho_{l0} \nabla \cdot \bfu = &\rho_l' \bmass'_t +\bmass' \rho'_{l \tspace t} + \rho_{l0} \bfu \cdot \nabla\bmass' \\
			- (1-\bmass_0) \bfu \cdot \nabla \rho_l' &  - (1-\bmass_0)\rho_l' \nabla \cdot \bfu + \rho_{l0}\bmass' \nabla \cdot \bfu,
		\end{aligned}
	\end{equation}
	where the first-order terms are on the left hand-side, second-order terms are collected on the right-hand side and higher-order terms have been neglected, similarly to before. By using linear approximations \eqref{eq: state linear}--\eqref{eq: bubbles linear}, we approximate these terms as follows:
	\begin{align}
		\rho_l' \bmass'_t &\approx \tfrac{1}{c^2} p'\bmass'_t, \qquad \qquad	\bmass' \rho_{l \tspace t}' \approx \tfrac{1}{c^2}\bmass' p'_t,\\
		\rho_{l0} \bfu \cdot \nabla\bmass' &\approx \tfrac{\rho_0}{1-\bmass_0} \bfu \cdot \nabla\bmass', \\
		- (1-\bmass_0)\bfu \cdot \nabla \rho_l' &\approx- \tfrac{1-\bmass_0}{c^2}\bfu \cdot \nabla p' =\tfrac{\rho_0 (1-\bmass_0)}{2 c^2} (\bfu \cdot \bfu)_t , \\
		- (1-\bmass_0)\rho_l' \nabla \cdot \bfu &\approx \tfrac{1-\bmass_0}{c^2} p' \left( \tfrac{1-\bmass_0}{\rho_0 c^2} p'_t- \tfrac{1}{1-\bmass_0} \bmass'_t \right) = \tfrac{(1-\bmass_0)^2}{2 \rho_0 c^4} (p'^2)_t - \tfrac{1}{c^2} p' \bmass'_t, \\
		\rho_{l0}\bmass' \nabla \cdot \bfu &\approx \rho_{l0}\bmass' \left( - \tfrac{1-\bmass_0}{\rho_0 c^2} p'_t + \tfrac{1}{1-\bmass_0} \bmass'_t \right) = \tfrac{\rho_0}{2(1-\bmass_0)^2} (\bmass'^2)_t  - \tfrac{1}{c^2}\bmass' p'_t .
	\end{align}
	Altogether, we arrive at
	\begin{equation} \label{eq: cons_mass2}
		\begin{aligned} 
			-&\tfrac{\rho_0}{1-\bmass_0} \bmass'_t + (1-\bmass_0) \rho_{l \tspace t}' + \rho_0 \nabla \cdot \bfu \\& =\tfrac{\rho_0}{(1-\bmass_0)} \bfu \cdot \nabla\bmass' + \tfrac{\rho_0 (1-\bmass_0)}{2 c^2}  (\bfu \cdot \bfu)_t 
			 + \tfrac{(1-\bmass_0)^2}{2 \rho_0 c^4} (p'^2)_t + \tfrac{\rho_0}{2(1-\bmass_0)^2} (\bmass'^2)_t.
		\end{aligned}	
	\end{equation}
	By substituting the nonlinear state equation \eqref{eq: state equation} into \eqref{eq: cons_mass2}, we obtain the following approximate version of \eqref{eq: continuum density}:
	\begin{equation}\label{eq: mass2}
		\begin{aligned} 
			- \tfrac{\rho_0}{1-\bmass_0} \bmass'_t + & \tfrac{1-\bmass_0}{c^2} p'_t - \tfrac{(1-\bmass_0)^2}{\rho_0 c^4} \tfrac{B}{2A} (p'^2)_t - \tfrac{\gamma (1-\bmass_0)^2}{\rho_0 c^4} \left( \tfrac{1}{c_v} - \tfrac{1}{c_p} \right) p'_{tt} \\
			=&  - \rho_0 \nabla \cdot \bfu + \tfrac{(1-\bmass_0)^2}{2 \rho_0 c^4} (p'^2)_t + \tfrac{\rho_0}{2(1-\bmass_0)^2}(\bmass'^2)_t \\
			&+\tfrac{\rho_0}{1-\bmass_0} \bfu \cdot \nabla\bmass' + \tfrac{\rho_0 (1-\bmass_0)}{2 c^2} (\bfu \cdot \bfu)_t . 
		\end{aligned}
	\end{equation}
	\subsubsection{Combining the simplified equations} Analogously to the derivation of nonlinear acoustic models for thermoviscous fluids, we apply the divergence operator to \eqref{eq: momentum3} and a time derivative to \eqref{eq: mass2}. Subtracting these two resulting equations 
	leads to the cancellation of the $\rho_0 \left( \nabla \cdot \bfu_t \right)$ terms and yields
	\begin{equation} \label{eq: merged}
		\begin{aligned}
			&\begin{multlined}[t]  \Delta p' - \tfrac{1-\bmass_0}{c^2} p'_{tt} +  \tfrac{\rho_0}{1-\bmass_0} \bmass'_{tt} + \tfrac{(1-\bmass_0)^2}{\rho_0 c^4} \tfrac{B}{2A} (p'^2)_{tt} \\
			 + \tfrac{\gamma (1-\bmass_0)^2}{\rho_0 c^4} \left( \tfrac{1}{c_v} - \tfrac{1}{c_p} \right) p'_{ttt}-\tfrac{1-\bmass_0}{2 \rho_0 c^2} \Delta p'^2+ \tfrac{(1-\bmass_0)^2}{2 \rho_0 c^4} (p'^2)_{tt} \end{multlined} \\
			=\,&\begin{multlined}[t]   - \frac{1}{1-\bmass_0} \nabla \cdot \left(\bmass' \nabla p' \right)  -  \tfrac{\left( \frac{4}{3} \mu + \mu_b \right)}{n_0} \Delta n'_t - \tfrac{\rho_0}{2} \Delta \left( \bfu \cdot \bfu \right) \\- \tfrac{\rho_0}{2(1-\bmass_0)^2} (\bmass'^2)_{tt}  
			 - \tfrac{\rho_0}{1-\bmass_0} \left( \bfu \cdot \nabla\bmass' \right)_t - \tfrac{\rho_0 (1-\bmass_0)}{2 c^2} (\bfu \cdot \bfu)_{tt}. \end{multlined}
		\end{aligned}
	\end{equation}
	We next make use of the relation 
$
		\bmass_t'= \frac{(1-\bmass_0)^2}{\rho_0 c^2} p_t' - \frac{1-\bmass_0}{n_0} n_t' 
$
	that results from subtracting \eqref{eg: continuity linear} from \eqref{eq: bubbles linear}. Together with \eqref{eq: momentum linear}, we obtain
	\begin{equation}
		\begin{aligned}
			- \tfrac{1}{1-\bmass_0} \nabla \cdot \left(\bmass' \nabla p' \right) - \tfrac{\rho_0}{1-\bmass_0} \left( \bfu\cdot \nabla\bmass' \right)_t = -  \tfrac{1}{1-\bmass_0} \bmass' \Delta p' + \tfrac{\rho_0 (1-\bmass_0)}{c^2} \bfu \cdot \bfu_{tt} - \rho_0 \bfu \cdot \Delta \bfu. 
		\end{aligned}
	\end{equation}
	%\begin{align}
	%	- \tfrac{1}{1-\bmass_0} \nabla \cdot \left(\bmass' \nabla p' \right) - \tfrac{\rho_0}{1-\bmass_0} \left( \bfu\cdot \nabla\bmass' \right)_t 
		%&= - \frac{1}{(1-\bmass_0)} \nabla \bmass' \cdot \nabla p'-  \frac{1}{(1-\bmass_0)} \bmass' \Delta p' - \frac{\rho_0}{(1-\bmass_0)} u'_t \cdot \nabla\bmass' -\frac{\rho_0}{(1-\bmass_0)} u' \cdot \nabla\bmass'_t  \\
		%& = - \frac{1}{(1-\bmass_0)} \nabla \bmass' \cdot \nabla p'-  \frac{1}{(1-\bmass_0)} \bmass' \Delta p' +\frac{1}{(1-\bmass_0)} \nabla p' \cdot \nabla\bmass' -\frac{\rho_0}{(1-\bmass_0)} u' \cdot \nabla\bmass'_t  \\
		%& = -  \frac{1}{1-\bmass_0} \bmass' \Delta p' -\frac{\rho_0}{1-\bmass_0} \bfu \cdot \nabla\bmass'_t  \\
		%& = -  \frac{1}{(1-\bmass_0)} \bmass' \Delta p' -\frac{\rho_0}{(1-\bmass_0)} u' \cdot \nabla \left( \frac{(1-\bmass_0)^2}{\rho_0 c^2} p_t' - \frac{(1-\bmass_0)}{n_0} n_t' \right)   \\
		%& = -  \frac{1}{(1-\bmass_0)} \bmass' \Delta p' - u' \cdot \left( \frac{(1-\bmass_0)}{c^2} \nabla p_t' - \frac{\rho_0}{n_0} n_t' \right)   \\
		%& = -  \frac{1}{(1-\bmass_0)} \bmass' \Delta p' - u' \cdot \left( \frac{(1-\bmass_0)}{c^2}\left( -\rho_0 u'_{tt}\right)  - \frac{\rho_0}{n_0} \left( - n_0 \Delta u' \right)  \right)   \\
		%= -  \tfrac{1}{1-\bmass_0} \bmass' \Delta p' + \tfrac{\rho_0 (1-\bmass_0)}{c^2} \bfu \cdot \bfu_{tt} - \rho_0 \bfu \cdot \Delta \bfu. %\\
		%& = -  \frac{1}{(1-\bmass_0)} \bmass' \Delta p' + \rho_0 u' \cdot \left( \frac{(1-\bmass_0)}{c^2} u'_{tt} - \Delta u' \right).
	%\end{align}
	This approach yields the following replacement for equation \eqref{eq: merged}:
	\begin{equation}\label{eq: merged1}
		\begin{aligned} 
		 &\begin{multlined}[t]	\Delta p' - \tfrac{1-\bmass_0}{c^2} p'_{tt} +  \tfrac{\rho_0}{1-\bmass_0} \bmass'_{tt} + \tfrac{(1-\bmass_0)^2}{\rho_0 c^4} \tfrac{B}{2A} (p'^2)_{tt} \\ \hspace*{1cm}
			 +  \tfrac{\gamma (1-\bmass_0)^2}{\rho_0 c^4} \left( \tfrac{1}{c_v} - \tfrac{1}{c_p} \right) p'_{ttt} -\tfrac{(1-\bmass_0)}{2 \rho_0 c^2} \Delta p'^2 +\tfrac{(1-\bmass_0)^2}{2 \rho_0 c^4} (p'^2)_{tt}  \end{multlined} \\
			=& \begin{multlined}[t] -  \tfrac{1}{1-\bmass_0} \bmass' \Delta p' + \rho_0 \bfu \cdot \left( \tfrac{1-\bmass_0}{c^2} \bfu_{tt} - \Delta \bfu \right)   -  \tfrac{ \frac{4}{3} \mu + \mu_b }{n_0} \Delta n'_t \\ \hspace*{1cm}
				- \tfrac{\rho_0}{2} \Delta \left( \bfu \cdot \bfu \right) - \tfrac{\rho_0 (1-\bmass_0)}{2 c^2} (\bfu \cdot \bfu)_{tt}  - \tfrac{\rho_0}{2(1-\bmass_0)^2} (\bmass'^2)_{tt}. \end{multlined}
		\end{aligned}
	\end{equation}
	Additionally, we use the relations between time and spatial derivatives according to the linear wave equation for pressure and acoustic velocity that are given by
	\begin{align} \label{eq: lin pressure}
		\Delta p'&\approx\tfrac{1}{c^2} p'_{tt}, \qquad \Delta p'^2\approx\tfrac{1}{c^2} (p'^2)_{tt},\\
		\Delta \bfu&\approx\tfrac{1}{c^2} \bfu_{tt}, \qquad \Delta (\bfu \cdot \bfu)\approx\tfrac{1}{c^2} (\bfu \cdot \bfu)_{tt}. \label{eq: lin velocity}
	\end{align}
	Employing \eqref{eq: lin pressure} and \eqref{eq: lin velocity} within the right-hand side of \eqref{eq: merged1} yields 
%	\begin{equation} \label{eq: merged2}
%		\begin{aligned}
%			\Delta p' -& \frac{1-\bmass_0}{c^2} p'_{tt} +  \frac{\rho_0}{1-\bmass_0} \bmass'_{tt} + \frac{(1-\bmass_0)^2}{\rho_0 c^4} \frac{B}{2A} (p'^2)_{tt} + \frac{\gamma(1-\bmass_0)^2}{\rho_0 c^2} \left( \tfrac{1}{c_v} - \tfrac{1}{c_p} \right) \Delta p'_{t} \\
%			=&    -  \frac{1}{1-\bmass_0} \bmass' \Delta p' +\frac{(1-\bmass_0)\bmass_0}{2 \rho_0 c^4} (p'^2)_{tt} -  \frac{\left( \frac{4}{3} \mu + \mu_b \right)}{n_0} \Delta n'_t  \\
%			&  +\frac{\rho_0 \bmass_0}{c^2} \left( \bfu\cdot \bfu \right)_t  - \frac{\rho_0}{c^2} (\bfu \cdot \bfu)_{tt} - \frac{\rho_0}{2(1-\bmass_0)^2} (\bmass'^2)_{tt}.
%		\end{aligned}
%	\end{equation}
%	By collecting similar terms, we obtain
	\begin{equation}\label{eq: merged3}
		\begin{aligned} 
			\Delta p' -& \frac{1-\bmass_0}{c^2} p'_{tt} +  \frac{\rho_0}{1-\bmass_0} \bmass'_{tt} +  \frac{\tilde{\beta}}{\rho_0 c^4}  ( p'^2)_{tt} +\frac{b}{c^2} \Delta p'_t  \\
			=& \begin{multlined}[t]   -  \frac{1}{1-\bmass_0} \bmass' \Delta p' -  \frac{\frac{4}{3} \mu + \mu_b }{n_0} \Delta n'_t  - \frac{\rho_0}{2(1-\bmass_0)^2} (\bmass'^2)_{tt} \\ \hspace*{3cm}
				+\frac{\rho_0 \bmass_0}{c^2} \left( \bfu \cdot \bfu \right)_t  - \frac{\rho_0}{c^2} (\bfu \cdot \bfu)_{tt} \end{multlined}
		\end{aligned}
	\end{equation}
	with $b=\frac{\gamma (1-\bmass_0)^2}{\rho_0} \left( \frac{1}{c_v} - \frac{1}{c_p} \right)$ and $\tilde{\beta}=(1-\bmass_0)^2 \frac{B}{2A}- \frac{(1-\bmass_0)\bmass_0}{2}$. 
	
	\subsubsection{Nonlinear acoustic equations for bubbly fluids} To arrive at bubbly media counterparts of classical models of nonlinear acoustics, we now consider the case that the bubble number density is constant in the equilibrium state, such that $n=n_0$ and hence $\bmass=n_0v$. 
	Additionally, we assume that the mean volume fraction occupied by gas is negligible $(\bmass_0 \approx 0)$, allowing us to approximate $\rho_0 \approx \rho_{l0}$. \vspace*{2mm}
	
	\noindent $\bullet$ \emph{Nonlinear in $v'$}. Considering all terms up to second order for $p'$ and $\bmass'=n_0v$, we can reduce \eqref{eq: merged3} to
	\begin{equation} \label{eq: merged reduced second} 
			\frac{1}{c^2}p'_{tt} \negthickspace-  \negmedspace \Delta p' \negthickspace -  \negmedspace \frac{b}{c^2 }  \Delta p'_t \negthickspace -  \negmedspace \frac{1}{\rho_0 c^4}\tfrac{B}{2A} (p'^2)_{tt} \negthickspace- \negmedspace \frac{\rho_0}{c^2} (\bfu \negthinspace \cdot  \negthinspace\bfu)_{tt} \negthickspace =  \rho_0 n_0 v'_{tt} \negmedspace +  \negmedspace n_0 v' \Delta p' \negthickspace +  \negmedspace\frac{\rho_0 n_0^2}{2} (v'^2)_{tt},
	\end{equation}
	where $b=\frac{\gamma}{\rho_0} \left( \frac{1}{c_v} - \frac{1}{c_p} \right)$ is the sound diffusivity.
	% and $\overline{\beta}= \frac{B}{2A}$. 
	This is an inhomogeneous version of the Kuznetsov equation~\cite{kuznetsov1971equations} with nonlinear coupling terms on the right-hand side. With the commonly used approximation 
$
		\| \bfu\| = \sqrt{\bfu \cdot \bfu } \approx \left| \tfrac{1}{c \rho_0} p'\right| $
	(valid if the propagation distance is much larger than a wavelength), we can reduce equation \eqref{eq: merged reduced second} to the inhomogeneous Westervelt equation  given by
	\begin{align} \label{eq: merged reduced5}
		\frac{1}{c^2 }p'_{tt} - \Delta p' -  \frac{b}{c^2 }  \Delta p'_t - \frac{\beta}{\rho_0 c^4} (p'^2)_{tt}=  \rho_0 n_0 v'_{tt} + n_0 v' \Delta p' + \frac{\rho_0 n_0^2}{2} (v'^2)_{tt}
	\end{align}
	%with $b=\frac{\gamma}{\rho_0} \left( \frac{1}{c_v} - \frac{1}{c_p} \right)$ and $\beta=\frac{B}{2A}+1$. \\
	with $\beta=\frac{B}{2A}+1$. \vspace*{2mm}
	
	\noindent $\bullet$ \emph{Linear in $v'$}. Retaining only first-order terms for $\bmass=n_0 v$ permits further reduction of equation \eqref{eq: merged3} to 
	\begin{align} \label{eq: merged reduced3}
		\frac{1}{c^2 }p'_{tt} -  \Delta p' - \frac{ b}{c^2 }  \Delta p'_t - \frac{1}{\rho_0 c^4} \tfrac{B}{2A}(p'^2)_{tt}- \frac{\rho_0 }{c^2 }(\bfu \cdot \bfu)_{tt}   =  \rho_0 n_0 v'_{tt}.
	\end{align}
	%with $b=\frac{\gamma}{\rho_0} \left( \frac{1}{c_v} - \frac{1}{c_p} \right)$ and $\tilde{\beta}=\frac{B}{2A}$. 
	Equation \eqref{eq: merged reduced3} is an inhomogeneous Kuznetsov equation with the right-hand side terms that are linear in $v'$. By reasoning as before, we obtain the following version of the inhomogeneous Westervelt equation:
	\begin{align} \label{eq: merged reduced4}
		\frac{1}{c^2 }p'_{tt} - \Delta p' -  \frac{b}{c^2 }  \Delta p'_t - \frac{\beta}{\rho_0 c^4} (p'^2)_{tt}=   \rho_0 n_0 v'_{tt},
	\end{align}
	which will be at the focal point of our theoretical and numerical investigations. 
	%with $b=\frac{\gamma}{\rho_0} \left( \frac{1}{c_v} - \frac{1}{c_p} \right)$ and $\beta=\frac{B}{2A}+1$. 
	Observe that equation \eqref{eq: merged reduced4} has been derived by taking into account all terms up to second order for $p'$ but retaining only terms that are linear in $v'$.  
	\subsection{Coupling to microbubble dynamics} \label{sec: bubble models}
		\begin{wrapfigure}{r}{5cm}
			\vspace{-0.48cm}
			\includegraphics{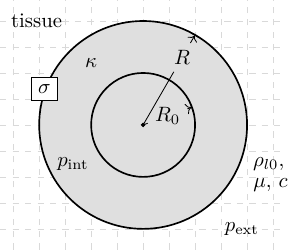}
		%\begin{figure}[h]
%		%	\scalebox{0.6}{
%			\begin{tikzpicture}[scale=0.44]
%				\tikzsetnextfilename{bubble}
%				\draw[help lines, color=gray!30, dashed] (-5.5, -4.8) grid (5.5, 4.8);
%				
%				
%				\draw [thick, fill = gray!25] circle [radius=2]
%				circle [radius=4];
%				
%				%\draw[double = gray!30, double distance=1.2cm, opacity=0.2] (0,0) circle [radius=3];
%				
%				% Short: polar coordinate and "syntax"
%				\draw[->|] (0,0) to[node on line, " $ \Rn$"] (20:2);
%				% Long:  rotate around and literal node after to
%				\draw[->|, rotate around={60:(0,0)}] % long form (short: "$4$" near end)
%				(0,0) to[node on line] node[near end]{ $R$} (4,0);
%				
%				\draw[fill=black](0,0) circle [radius=2pt];
%				\node at (-4.1,4) { tissue};
%				\node at (3.8,-4) { $\pext$};
%				\node at (-2.7,-1.5) { $\pint$};
%				\node at (4.9,-1.5) { $\rho_{l0},\,$};
%				\node at (4.9,-2.3) { $\mu, \, c$};	
%				\node[draw, fill=white] at (-3.8,1.4) { $\sigma$};				
%				\node at (-2,2.4) { $\kappa$};
%			\end{tikzpicture}
%			%}
		\vspace{-0.6cm}
		\caption{Single-microbubble dynamics; adapted from~\cite{lauterborn2023acoustic}} \label{Fig:Bubble}
		%\end{figure}
		\vspace{-0.2cm}
	\end{wrapfigure}
	Solving equation \eqref{eq: merged reduced4} requires an additional relation between the sound pressure and the microbubble volume and we therefore couple it with a Rayleigh--Plesset-type equation in terms of the bubble radius $R$ together with the volume-radius relation $v=\frac{4}{3}\pi R^3$. As already mentioned, many different models exist for describing the oscillating dynamics of microbubbles. We provide an overview here and refer to, e.g., the book~\cite{brennen_2013} and the review papers~\cite{versluis2020ultrasound, doinikov2011review} for further details.	\\
		 Starting from
	\begin{equation} \label{RP_starting}
			\rho_{l0}  \left[ R R_{tt} + \tfrac{3}{2} R_t^2 \right]	=\, \pint-\pext 
			%\vspace{-0.2cm}
	\end{equation}
	where we recall that $\rho_{l0}$ is the mean mass density of the liquid and $\pint$ the internal pressure, the simplest Rayleigh--Plesset equation is given by 
	\begin{equation} \label{RP_1a}
		\begin{aligned}
			\rho_{l0}  \left[ R R_{tt} + \tfrac{3}{2} R_t^2 \right] = \pb -4 \mu\frac{ R_t}{R} -p',
		\end{aligned}
	\end{equation}
	where $\pb= \pv-\pstat$, $\pv$ is the vapor pressure, $\pstat$ is the static ambient pressure, and $\mu$ is the  viscosity of the medium; see Figure \ref{Fig:Bubble}.
 A modified equation is given by
	\begin{equation} \label{RP_2}
		\begin{aligned}
			\rho_{l0} \left[ R R_{tt} + \tfrac{3}{2} R_t^2 \right]  = \pb - 4 \mu \dfrac{R_t}{R}- \dfrac{2 \sigma}{R}  -p',
		\end{aligned}
	\end{equation} 
	where $\sigma$ is the surface tension between the liquid and vapor/gas at the bubble wall. \\ %; see Figure~\ref{Fig:Bubble}.\\
	\indent A further extension of the ODE arises from taking bubble contents into considerations. A generalized Rayleigh--Plesset equation, also known as the RPNNP (which stands for Rayleigh--Plesset--Noltingk--Neppiras--Poritsky) equation is given by
	\begin{equation} \label{RP_3}
		\begin{aligned}
			\rho_{l0} \left[ R R_{tt} + \tfrac{3}{2} R_t^2 \right]
			=\, \pb-  4 \mu \dfrac{ \Rt}{R} -\dfrac{2 \sigma}{R}+\ppgn \left(\dfrac{\Rn}{R} \right)^{3 \kappa}-p';
		\end{aligned}
	\end{equation}
	see~\cite{rayleigh1917viii, noltingk1950cavitation}. Here $\ppgn$ is the (constant) gas pressure  inside the bubble at rest given by $p_{\textup{pgn}} = \frac{2 \sigma}{\Req} - \pb;$
	see ~\cite{lauterborn2023acoustic}. Furthermore, $\Rn$ in \eqref{RP_3} is the equilibrium radius, i.e., the radius of a bubble at rest and $\kappa \geq 1$ is the adiabatic index. \\[1mm] %; see Figure~\ref{Fig:Bubble}.  %Furthermore, $\varrho$ and $\mu$ are the (constant) mass density and the  viscosity of the (tissue) medium, respectively.    \\
%	\begin{center}
%		\begin{figure}[h]
%			%	\scalebox{0.6}{
%				\begin{tikzpicture}[scale=0.5]
%					\tikzsetnextfilename{bubble}
%					\draw[help lines, color=gray!30, dashed] (-5.9, -5.9) grid (5.9, 5.9);
%					
%					
%					\draw [thick, fill = gray!25] circle [radius=2]
%					circle [radius=4];
%					
%					%\draw[double = gray!30, double distance=1.2cm, opacity=0.2] (0,0) circle [radius=3];
%					
%					% Short: polar coordinate and "syntax"
%					\draw[->|] (0,0) to[node on line, " $ \Rn$"] (20:2);
%					% Long:  rotate around and literal node after to
%					\draw[->|, rotate around={60:(0,0)}] % long form (short: "$4$" near end)
%					(0,0) to[node on line] node[near end]{ $R$} (4,0);
%					
%					\draw[fill=black](0,0) circle [radius=2pt];
%					\node at (-4,4) { tissue};
%					\node at (3.8,-4) { $\pext$};
%					\node at (-2.7,-1.5) { $\pint$};
%					\node at (5.1,-1.5) { $\rho_{l0},\,$};
%					\node at (5.1,-2.3) { $\mu, \, c$};	
%					\node[draw, fill=white] at (-3.8,1.4) { $\sigma$};				
%					\node at (-2,2.4) { $\kappa$};
%				\end{tikzpicture}
%				%}
%			\caption{Single-microbubble dynamics;  adapted from~\cite{lauterborn2023acoustic}} \label{Fig:Bubble}
%		\end{figure}
%	\end{center}
	\noindent $\bullet$ \emph{Incorporating acoustic radiation}. Other versions of microbubble models arise from considering damping of the bubble dynamics by the sound radiated by the bubble. A model that incorporates sound radiation is given by (see~\cite{versluis2020ultrasound})
	\begin{equation} \label{RP_4}
		\begin{aligned}
			\rho_{l0}  \left[ R R_{tt} + \tfrac{3}{2} R_t^2 \right]
			=\, \pb-  4 \mu \dfrac{ \Rt}{R} -\dfrac{2 \sigma}{R}+\ppgn \left(\dfrac{\Rn}{R} \right)^{3 \kappa}\left(1- 3 \kappa \dfrac{R_t}{c}\right)-p'.
		\end{aligned}
	\end{equation}
	
	\noindent $\bullet$ \emph{Coated-bubble dynamics}. The coating of bubbles is known to stabilize them and prolong their existence; see, e.g.,~\cite{versluis2020ultrasound}. In this case, the size-dependent effective surface tension can given by $\sigma(R) = \chi \bigl( \tfrac{R^2}{\Req^2}-1 \bigr)$, 
	where $\chi$ is the shell elasticity. The dynamics of microbubbles with thin shells can then be described by
	\begin{equation} \label{RP_shell_1}
		\begin{aligned}
			\rho_{l0}  \left[ R R_{tt} + \tfrac{3}{2} R_t^2 \right]
			=\, \pb-  4 \mu \frac{ \Rt}{R} -\frac{2\sigma(R)}{R}+\ppgn \left(\frac{\Rn}{R} \right)^{3 \kappa}\left(1- 3 \kappa \frac{R_t}{c}\right)-4 \kappa_s \frac{R_t}{R^2}-p'
		\end{aligned}
	\end{equation}
	with $p_{\textup{pgn}} = \frac{2 \sigma_0}{\Req} - \pb$, where $\sigma_0$ is the surface tension of the bubble at rest and $\kappa_s$ is the surface dilatational viscosity of the shell; see~\cite[eq.\ (8)]{versluis2020ultrasound}.  \vspace*{1mm}
	%Further models that take the width of the shell into account can be found in, e.g.,~\cite[Appendix A]{morgan2000experimental} and the review paper~\cite{doinikov2011review}.

\noindent{\bf Generalized ODE}. Going forward, in the analysis, we assume $\rho_0 \approx \rho_{l0}$ and consider the following mathematically generalized equation for the dynamics of microbubbles:
	\begin{equation}
		\begin{aligned}
			\rho_0  \left[ R R_{tt} + \tfrac{3}{2} R_t^2 \right] \begin{multlined}[t]=h_0(R, \Rt) - p', \end{multlined}
		\end{aligned}
	\end{equation}
	where we impose local Lipschitz continuity on the function $\hzero$; see Theorem~\ref{Thm:LocalWellp_WestRP} for details. The assumptions allow having 
	\begin{equation} \label{def hzero}
		\begin{aligned}
			h_0(R, \Rt)=\begin{multlined}[t]  \pb -   \frac{4 \mu}{R}\Rt -\frac{2\sigma(R)}{R} +\ppgn \left(\frac{\Rn}{R} \right)^{3 \kappa}{\left(1- 3 \kappa_0 \frac{\Rt}{c}\right)} { -4 \kappa_s \dfrac{\Rt}{R^2} }
			\end{multlined}
		\end{aligned}
	\end{equation} 
	with $\sigma$ constant or $\sigma(R) = \chi \left( \tfrac{R^2}{\Req^2}-1\right)$ and $\kappa_s$, $\kappa_0 \in \R$ so that equation \eqref{RP_4} (with $\kappa_0 = \kappa$) and equation \eqref{RP_3} (with $\kappa_0=\kappa_s=0$) are covered by our theoretical results. The ODE will then be coupled to the damped Westervelt equation for the pressure, using the fact that $v'_{tt}=v_{tt}$ to arrive at the right-hand side. When considering this coupled system, $R=R(x,t)$ is a function of both time and space. Thus, the equation governing the dynamics of microbubbles is an ODE defined pointwise in space.

	\section{Analysis of the Westervelt--Rayleigh--Plesset system} \label{Sec:Analysis}
	In this section, we analyze the Westervelt--Rayleigh--Plesset system. For simplicity of notation, we drop $(\cdot)'$ when denoting fluctuating pressure quantities and use $p$ instead of $p'$, but we keep the notation $R=R_0 + R'$ for the total bubble radius. We consider the following initial boundary-value problem:
	\begin{equation} \label{ibvp_fractionalWestRP_general}
		\left \{ \begin{aligned}
			&((1+ 2k(x)p) \pt)_t - c^2 \Delta p - b \Delta \pt= \xi (R^3)_{tt} \quad \text{in } \Omega \times (0,T), \\[1mm]
			&(p, p_t)_{ \vert t=0} = (p_0, p_1), \quad p_{\vert \partial \Omega} =0,\\[1mm]
			& 		\rho_0  \left[ R R_{tt} + \tfrac{3}{2} R_t^2 \right] \begin{multlined}[t]=h_0(R, \Rt)-p   \qquad \text{in } \Omega \times (0,T), \end{multlined} \\[1mm]
			&  (R, \Rt)_{ \vert t=0} = (R_0, R_1),
		\end{aligned} \right.
	\end{equation}
	with $\xi \in \R$ (corresponding to $\frac43 \pi c^2 \eta$ in the derivation).
	The aim of this section is to determine the conditions on the initial pressure-microbubble data under which \eqref{ibvp_fractionalWestRP_general} has a unique solution.  
	\subsection*{Notation} \emph{Below we use $x \lesssim y$ to denote $x \leq C y$, where $C>0$ is a generic positive constant which may depend on the final time in such a manner that it tends to $+\infty$ as $T \rightarrow +\infty$.  We use the notation $\Honetwo=\, 	H^2(\Omega) \cap H_0^1(\Omega)$. 
	When denoting norms in Bochner spaces, we omit the temporal domain; for example, $\|\cdot\|_{L^p(L^q(\Om))}$ denotes the norm in $L^p(0,T; L^q(\Om))$.}\\[2mm]
 
We recall first a well-posedness result for the Westervelt equation for non-bubbly media with strong attenuation, which is needed to set up the analysis of the coupled problem. This form of Westervelt's equation has been extensively studied in the mathematical literature; see, e.g.,~\cite{kaltenbacher2009global, meyer2011optimal, kaltenbacher2022parabolic, kaltenbacher_rundell2021} and the references provided therein. One of the key aspects of its analysis is ensuring that the factor $1+2k p$ next to the second time derivative of the pressure stays  positive (i.e., that it does not degenerate), which can be achieved through sufficiently small (and smooth) data.
	\begin{proposition} \label{Prop:WellP_West}
		Assume that $\Omega \subset \R^d$, $d \in \{1, 2,3\}$ is a bounded and $C^{1,1}$-regular domain. Let $c$, $b>0$ and $k \in L^\infty(\Om)$. Furthermore, let $(p_0, p_1) \in  \Honetwo \times \Honezero$ and  $f \in L^2(0,T; \Ltwo)$.
		There exists data size $\delta=\delta(T)>0$,  such that if
		\begin{equation}\label{smallness_delta_Westervelt}
			\|p_0\|_{\Htwo}+\|p_1\|_{\Hone} + \|f\|_{L^2(\Ltwo)} \leq \delta,
		\end{equation}
		then there is a unique solution of 
		\begin{equation} \label{IBVP_West}
			\left \{ \begin{aligned}
				&((1+ 2k(x)p) \pt)_t - c^2 \Delta p - b \Delta \pt = f(x,t) \quad \text{in } \Omega \times (0,T), \\
				&p_{\vert \partial \Omega} =0, \quad  (p, p_t)_{ \vert t=0} = (p_0, p_1)
			\end{aligned} \right.
		\end{equation}
		in
				$\Xp =  L^\infty(0,T; \Honetwo) \cap W^{1, \infty}(0,T; \Honezero) \cap H^2(0,T; \Ltwo)$,
		such that $1+ 2k p \geq \gamma >0 \ \text{ in }\ \Omega \times (0,T)$
		for some $\gamma>0$.	Furthermore, 
		\begin{equation}
			\begin{aligned}
				\|p\|^2_{\Xp} \lesssim\|p_0\|^2_{\Htwo}+\|p_1\|^2_{\Hone} + \|f\|^2_{L^2(\Ltwo))} .
			\end{aligned}
		\end{equation}
	\end{proposition}
	\begin{proof}
		The proof follows analogously to that of \cite[Proposition 1]{kaltenbacher_rundell2021}; the only difference is that here smallness of data is imposed instead of the smallness of $\|k\|_{L^\infty(\Omega)}$. We thus omit the details.
	\end{proof}
	Under the assumptions of Proposition~\ref{Prop:WellP_West}, we can define the solution mapping 
	\[
	\calS: \LtwoTLtwo \rightarrow \Xp,\quad \text{ such that  }  \calS(f) = p,
	\]
	which will be employed in the analysis of the coupled problem. We will furthermore exploit Lipschitz continuity of this mapping in the following sense. Let $\fone$, $\ftwo \in  \LtwoTLtwo$ and denote $\pone = \calS(\fone)$ and $\ptwo = \calS(\ftwo)$. By Proposition~\ref{Prop:WellP_West}, 
	\[
	\|\pone\|_{\Xp}, \, \|\ptwo\|_{\Xp} \leq \mathcal{C}
	\]
	for some $\mathcal{C>0}$. Then it can be shown (analogously to the proof of contractivity in \cite[Proposition 1]{kaltenbacher_rundell2021}) that
	\begin{equation} \label{Lipschitz continuity p}
		\begin{aligned}
			\| \pone - \ptwo\|_{\Xp} \leq C(T, \mathcal{C}) \|\fone - \ftwo\|_{\LtwoLtwo};
		\end{aligned}
	\end{equation}
	for completeness, we provide the details in Appendix~\ref{Appendix: Lipschitz}. \\
\indent We approach the analysis of the Westervelt--Rayleigh--Plesset system by employing Banach's fixed-point theorem. To this end, we linearize the Rayleigh--Plesset equation, in the general spirit of~\cite{vodak2018mathematical, biro2000analysis}, while eliminating $p$ via the mapping $\calS$. More precisely, to analyze \eqref{ibvp_fractionalWestRP_general} under the assumptions of Proposition~\ref{Prop:WellP_West}, we set up the fixed-point mapping $\calT: \BR \ni \Rs \mapsto R$,
	where $R$ solves %the problem
	\begin{equation} \label{linearized_system}
		\left \{ \begin{aligned}
			& \Rtt =	h(\Rs, \Rst, \calS(f(\Rs))) \\
			&(R, \Rt) \vert_{t=0} = (R_0, R_1),
		\end{aligned} \right.
	\end{equation}
	with $f(\Rs) = \xi \left((\Rs)^3\right)_{tt}$ and the function $h$ given by
	\begin{equation} \label{def h}
		h(\Rs, \Rst, \calS(f(\Rs))) = \frac{1}{\Rs} \bigl(-\tfrac32 (\Rst)^2 +\rhoinv \hzero(\Rs, \Rst)- \rhoinv \calS(f(\Rs)) \bigr).
	\end{equation}
	Above, $R^*$ is taken from the set \vspace*{-1mm}
	\begin{equation} \label{ball_R}
		\begin{aligned}
		\BR \negthinspace = \negthinspace \big\{ \negthinspace \Rs \negthickspace \in \negthinspace C^{2} ([0,T]; \Linf) \negthinspace:&\, \|\Rstt\|_{\CLinf}  \negthickspace \leq \negthinspace M, \\ &\|\Rs\|_{\CLinf}  \negthickspace + \negthinspace \|\Rst\|_{\CLinf}  \negthickspace \leq m, \\[1mm]
		&\,  \|\Rs -R_0\|_{C(\Linf)} \leq \varepsilon_0, \,(\Rs, \Rst)_{\vert t=0}=(R_0, R_1) \big \}.
		\end{aligned}
	\end{equation}
	Above, $\varepsilon_0>0$  is small enough and will be set by the upcoming proof, together with small enough $m>0$ and an adequately calibrated $M>0$.  We note that the $\eps_0$ condition  is there to ensure the positivity of $R$, provided $R_0$ is positive. \\
	\indent  We make the following assumptions on the function $h_0$ in \eqref{def h}. We assume that for any $\Rs \in \BR$, and for any $\Rsone$, $\Rstwo \in \BR$,
	\begin{align}
		\|\hzero(\Rs, \Rst)\|_{\CLinf} \leq C&(m, \varepsilon_0)  \label{assumption1 h_0} \\
		\|\hzero(\Rsone, \Rsone_t)-\hzero(\Rstwo, \Rstwo_t)\|_{\CLinf} & \lesssim  \|\Rsone-\Rstwo\|_{\ConeLinf} \label{assumption2 h_0}
	\end{align}  
	We note that function $\hzero$ in \eqref{def hzero} satisfies these assumptions; for completeness, we provide the proof in Appendix~\ref{Appendix: Lipschitz}.
	\begin{theorem} \label{Thm:LocalWellp_WestRP} Assume that $\Omega \subset \R^d$, where $d \in \{2,3\}$, is a bounded and $C^{1,1}$-regular domain. Furthermore, assume that $c$, $b>0$,  $\rho_0>0$, $\xi \in \R$, and $k\in L^\infty(\Om)$, and let $(p_0, p_1) \in  \Honetwo \times \Honezero, \, (R_0, R_1) \in \Linf \times \Linf$,
		where
		\begin{equation}
			R_0(x) \geq \underline{R}_0 >0 \ \quad \text{a.e. in } \Omega.
		\end{equation}
		Let assumptions \eqref{assumption1 h_0} and \eqref{assumption2 h_0} on the function $\hzero$ hold. Then there exist pressure data size $\deltap>0$, bubble data size $\deltaR $, 
		and final time $\tilde{T}$,
		such that if
		\begin{equation}
			\|p_0\|_{\Htwo}+\|p_1\|_{\Hone} \leq \deltap, \quad \|R_0\|_{\Linf} + \|R_1\|_{\Linf} \leq \deltaR, \ \text{ and } \ T \leq \tilde{T},
		\end{equation}
		there is a unique $(p, R) \in \Xp \times \BR$ that solves \eqref{ibvp_fractionalWestRP_general}.
	\end{theorem}
	\begin{proof}
		As announced, we carry out the proof by relying on the Banach fixed-point theorem. Note that $\BR \neq \emptyset$ since the solution of 
		\[
		\begin{aligned}
			\Rtt=0, \quad (R, \Rt)\vert_{t=0}= (R_0, R_1)
		\end{aligned}
		\]
		belongs to $\BR$ as long as the final time and $\delta_R$ are small enough.\\
		\indent  Let $\Rs \in \BR$. For small enough $\varepsilon_0=\varepsilon_0( \underline{R}_0)$, $\Rs$ is bounded and positive:
		\begin{equation} \label{boundedness_Rs}
			0 < \ulR:=  \underline{R}_0-\varepsilon_0 \leq  R_0-\varepsilon_0 \leq \Rs(x,t) \leq \|R_0\|_{\Linf}+ \varepsilon_0 := \olR
		\end{equation}
		a.e.\ in $\Om \times(0,T)$. 
		Toward checking the assumptions of Proposition~\ref{Prop:WellP_West}, we note that
		\begin{equation} \label{est_f_further}
			\begin{aligned}
				\|f(\Rs)\|_{L^2(\Ltwo)} \negthickspace
				%	=& \, \xi\| ((\Rs)^3)_{tt}\|_{L^2(\Ltwo)} \\
				%	=&\, \xi\| 3 (\Rs)^2 \Rstt+ 6 \Rs (\Rst)^2\|_{L^2(\Ltwo)} \\
				\leq&\, 3\xi\|\Rs\|^2_{\CLinf} \negthickspace \sqrt{T}\|\Rstt\|_{C(\Ltwo)} \negthickspace+ \negthinspace 6 \xi\|\Rs \negthinspace\|_{\CLinf}\negthickspace \sqrt{T}\|\Rst\|^2_{C(\Ltwo)} \\
				\leq&\,\sqrt{T} C(\Omega) (M+m)m^2.
			\end{aligned}
		\end{equation}
		Therefore, the smallness condition \eqref{smallness_delta_Westervelt} can be fulfilled by making $\deltap$ and $T$ small enough so that
	$
				\delta_p +	\sqrt{T} C(\Omega) (M+m)m^2\leq \delta$.
		Thus, by Prop.~\ref{Prop:WellP_West}, the mapping $\calS$ is well-defined (and, in turn, problem \eqref{linearized_system}) and we have $p= \calS(f(\Rs))$.  The remainder of the proof is dedicated to showing that $\calT$ is a strictly contractive self-mapping. \vspace*{2mm}

		%\fbox{ $\bullet$ The self-mapping property}
		\noindent \underline{The self-mapping property}: 
		From \eqref{linearized_system}, we have the following bound:
		\begin{equation} \label{est_R}
			\begin{aligned}
				&\|\Rtt\|_{\CLinf}
				%		\leq&\,\begin{multlined}[t] \|R_0 +R_1 t + \int_0^t \int_0^s h(\Rs, \Rst, p)(s)\dst\|_{\CLinf}\\+ \|R_1 + \int_0^t  h(\Rs, \Rst, p)(t)\dt\|_{\CLinf}+ \| h(\Rs, \Rst, p\|_{\CLinf}
					%		\end{multlined}
				%		\\
				\leq\, \|h(\Rs, \Rst, \calS(f(\Rs)))\|_{\CLinf}.
			\end{aligned}
		\end{equation}
		To estimate the $h$ term further, we use the lower bound for $\Rs$ established in \eqref{boundedness_Rs}:
		\begin{equation} 
			\begin{aligned}
				\|h(\Rs, \Rst, \calS(f(\Rs))) \|_{\CLinf} \negthickspace
				\leq\begin{multlined}[t] \negthickspace \tfrac{1}{ \olR}\bigl \|  \scalebox{0.75}[1.0]{\( - \)}\tfrac32 (\Rst)^2 \negthickspace+\tfrac{1}{\rho_0} \hzero(\Rs \negthickspace, \Rst) -\tfrac{1}{\rho_0} \calS(f(\Rs)) \bigr  \|_{\CLinf} 	\end{multlined}.
			\end{aligned}
		\end{equation}
		On account of the fact that $\Rs \in \BR$ and assumption \eqref{assumption1 h_0} made on $\hzero$, we then have
		\begin{equation} \label{est_h}
			\begin{aligned}
				\|h(\Rs \negthickspace, \Rst \negthickspace, \calS(f(\Rs))) \|_{\CLinf} \negthinspace 
				\lesssim  \negthinspace  \|\Rst\|_{\Linf}^2 \negthickspace + \negthinspace  C(m, \varepsilon_0) \negthinspace + \negthinspace \|\calS(f(\Rs))\|_{\CLinf}.	
			\end{aligned}
		\end{equation}
		To estimate the $\calS$ term further, we employ Proposition~\ref{Prop:WellP_West} and the embedding $C([0,T]; \Linf) \hookrightarrow \Xp$:
		\[
		\begin{aligned}
			\|\calS(f(\Rs))\|_{\CLinf} \leq&\, C(\Om) \|\calS(f(\Rs))\|_{\Xp}\\
			\leq&\, C(\Om, T) \left(\|p_0\|_{\Htwo}+\|p_1\|_{\Hone}+\|f(\Rs)\|_{\LtwoLtwo} \right).
		\end{aligned} 
		\]
		Then taking into account the bound on $f$ derived in \eqref{est_f_further} yields
		\begin{equation} \label{est Sf}
			\|\calS(f(\Rs))\|_{\CLinf}  \leq  C(\Omega, T) \bigl(\deltap+\sqrt{T}(M+m)m^2 \bigr).
		\end{equation}
		Therefore, by employing \eqref{est Sf} in \eqref{est_h} and then \eqref{est_R}, we conclude that
		\begin{equation}
			\begin{aligned}
				\|\Rtt \|_{\CLinf} 
				%	\lesssim&\,  \|R_0\|_{L^\infty(\Linf)}+\|R_1\|_{L^\infty(\Linf)}+ \|h(\Rs, \Rst, p)\|_{L^2(\Linf)}\\
				%		\lesssim&\,\begin{multlined}[t]	\|R_0\|_{\Linf}+T\|R_1\|_{\Linf}+T \Big \{ \frac{1}{\olR} \|\Rst\|_{\CLinf}^2  + \|\pb\|_{\CLinf} + C_1\exp(C_2T) \delta \\+|\ppgn| \left(\frac{\Rn}{\olR} \right)^{3 \gamma}+\frac{1}{\olR} +\frac{1}{\olR}\|\Rst\|_{\CLinf} 	\Big\} \end{multlined}	 \\
				\lesssim\,  1+\deltap+m^2+\sqrt{T} (M+m)m^2	 
				\leq\, M,
			\end{aligned}
		\end{equation}
		provided $M>0$ is large enough and $T>0$ small enough relative to it. Next, since
		\begin{equation} \label{RP_eq_solved}
			\begin{aligned}
			%	\Rt=&\, 
		%		R_1  +  \int_0^t h(\Rs, \Rst, \calS(f(\Rs)))(s)\ds,\\
				R=&\, 
					R_0 +R_1 t + \int_0^t \int_0^s h(\Rs, \Rst, \calS(f(\Rs)))(\ts)\dtsds,
			\end{aligned}
		\end{equation}
		we find that 
		\begin{equation}
			\begin{aligned}
				&\|R \|_{\CLinf} + \|\Rt\|_{\CLinf} \\
				\leq&\, \|R_0\|_{\Linf} + (1+T)\|R_1\|_{\Linf} + (T+T^2) \|h(\Rs, \Rst, \calS(f(\Rs)))\|_{\CLinf} \\
				\lesssim&\,\begin{multlined}[t] (1+T)\deltaR+ (T+T^2) \bigl [1+m^2+\deltap+	\sqrt{T}  (M+m)m^2 \bigr], \end{multlined}
			\end{aligned}
		\end{equation}
		which can be made smaller than $m$ by decreasing $\deltaR$ and $T$.  To conclude that $R \in \BR$, it remains to estimate $R-R_0$. To this end, the expression for $R$ in \eqref{RP_eq_solved} yields
		\begin{equation}
			\begin{aligned}
				\|R-R_0\|_{C(\Linf)} \leq T \|R_1\|_{\Linf}+T^2\|h(\Rs, \Rst, \calS(f(\Rs))\|_{C (\Linf)}.
			\end{aligned}
		\end{equation} 
		Therefore, reducing $T$ allows us to conclude that $\|R-R_0\|_{C(\Linf)} \negthickspace \leq \varepsilon_0$,
		and, % in turn,  
		$R \in \BR$. \vspace*{2mm}
		
		\noindent \underline{Strict contractivity}: To prove strict contractivity of $\calT$, we take $R^{(1),*}$, $R^{(2),*} \in \BR$ and note that the difference $\tR= \calT (\Rsone)-\calT(\Rstwo)$ solves 
		\begin{equation} \label{RP_diff}
			\begin{aligned}
				\tR_{tt}=&\, \begin{multlined}[t] h(\Rsone, \Rsone_t, \calS(f(\Rsone)))-h(\Rstwo, \Rstwo_t, \calS(f(\Rstwo)))  \end{multlined}	
			\end{aligned}
		\end{equation}
		with zero initial conditions. Thus we immediately have 
		\begin{equation} \label{contractivity_est}
			\begin{aligned}
				\|\tR\|_{\CtwoLinf} \negthickspace
				\lesssim \negthickspace \| h(\Rsone \negthickspace, \Rsone_t \negthickspace, \calS(f(\Rsone)))-h(\Rstwo \negthickspace, \Rstwo_t \negthickspace, \calS(f(\Rstwo)))\|_{\CLinf}
			\end{aligned}
		\end{equation}
		and with $\tR^*= \Rsone-\Rstwo$ we can rewrite the difference of $h$ terms as follows:
		\begin{equation} \label{diff h}
			\begin{aligned}
				&h(\Rsone, \Rsone_t, \calS(f(\Rsone)))-h(\Rstwo, \Rstwo_t, \calS(f(\Rstwo))) \\
				=&\,  \begin{multlined}[t] -\tfrac{\tR^*}{\Rsone \Rstwo}\left [ - \tfrac32 (\Rsone_t)^2 +\rhoinv \hzero(\Rsone, \Rsone_t) \right]\\
					+ \tfrac{1}{\Rstwo}\left[\scalebox{0.75}[1.0]{\( - \)} \tfrac32 \tR^*_t(\Rsone_t \negthickspace+ \Rstwo_t)  + \negthinspace \rhoinv \hzero(\Rsone \negthickspace, \Rsone_t)-\rhoinv \hzero(\Rstwo\negthickspace, \Rstwo_t) \right] \\
					+\tfrac{\tRs}{\rho_0 \Rsone \Rstwo} \calS(f(\Rstwo)) -\tfrac{1}{\rho_0 \Rsone}\left( \calS(f(\Rsone))-\calS(f(\Rstwo))\right) .
				\end{multlined}	
			\end{aligned}
		\end{equation}
		Thanks to the Lipschitz continuity of the pressure field stated in \eqref{Lipschitz continuity p} and the embedding $C([0,T]; \Linf) \hookrightarrow \Xp$, we know that
		\begin{equation}
			\begin{aligned}
				\|\calS(f(\Rsone))-\calS(f(\Rstwo))\|_{\CLinf} 
				%\lesssim&\, \|f(\Rsone)-f(\Rstwo)\|_{\LtwoLtwo} \\
				\lesssim&\, \sqrt{T} \|f(\Rsone)-f(\Rstwo)\|_{\CLinf}.
			\end{aligned}
		\end{equation}
		We can further see the difference of $f$ terms as
		\begin{equation}
			\begin{aligned}
				f(\Rsone)-f(\Rstwo)
			%	=&\, \begin{multlined}[t] 3\xi \bigl[	(\Rsone)^2 \Rsone_{tt}-(\Rstwo)^2 \Rstwo_{tt} \bigr]+ 6 \xi \bigl[\Rsone (\Rsone_t)^2-\Rstwo (\Rstwo_t)^2\bigr] 	\end{multlined}\\
				=&\, \begin{multlined}[t] 3\xi \bigl[	\tRs(\Rsone+\Rstwo) \Rsone_{tt}+(\Rstwo)^2 \tRs_{tt} \bigr]\\\hspace*{1cm}+ 6 \xi \bigl[\tRs (\Rsone_t)^2+\Rstwo \tRs_t(\Rsone_t+\Rstwo_t)\bigr] .	\end{multlined} 
			\end{aligned}
		\end{equation}
		From here, using the fact that $\Rsone$, $\Rstwo \in \BR$, we obtain
		\[
		\|f(\Rsone)-f(\Rstwo)\|_{\CLinf} \leq\, C(M) m \|\tRs\|_{\CtwoLinf}.
		\]
		Combining this bound with the continuity of $h_0$ in the sense of 
		\begin{equation}
			\begin{aligned}
				\|\hzero(\Rsone, \Rsone_t)-\hzero(\Rstwo, \Rstwo_t)\|_{\CLinf} \lesssim&\, \|\Rsone-\Rstwo\|_{\ConeLinf} \\
				\lesssim&\, (T^2+T)\|\Rsone_{tt}-\Rstwo_{tt}\|_{\CLinf}
			\end{aligned}
		\end{equation}
		and estimating the remaining terms on the right-hand side of \eqref{diff h} in a similar manner leads to 
			\begin{equation} \label{est diff h}
			\begin{aligned}
				&\| h(\Rsone,  \Rsone_t, \calS(f(\Rsone)))-h(\Rstwo, \Rstwo_t, \calS(f(\Rstwo)))\|_{\CLinf} \\
				\leq&\ \sqrt{T}\cdot C( \ulR, \olR, m, M, \tilde{T}) \|\tR^*\|_{\CtwoLinf}.
			\end{aligned}
		\end{equation}
		 Then using  \eqref{est diff h} in \eqref{contractivity_est} yields the estimate
	$
				\|\tR\|_{\CtwoLinf} \lesssim\,  \sqrt{T}\|\tR^*\|_{\CtwoLinf}$. Therefore, the mapping $\calT$ is strictly contractive with respect to the norm in the space $C^2([0,T]; \Linf)$ provided final time $T$ is sufficiently small. An application of Banach's fixed-point theorem yields the desired result.
	\end{proof}
\begin{remark}[On nonlocal acoustic attenuation]	In complex media, such as soft tissue, sound attenuation may be more accurately modeled using time-fractional damping. The modeling and analysis framework employed here extends to that setting as well. For completeness, we provide details in Appendix~\ref{Appendix: Fractional}.
	\end{remark}
	%%%%%%%%%%%%%%%%%%%%%%%%%%%%%%%		
	\section{Numerical study} \label{Sec:Numerics}
In this section, we {conduct a further numerical study of the problem. In particular, we wish to explore the empirical behavior of microbubbles. We start by examining the dynamics of a single microbubble with a sinusoidal driving pressure. Afterwards, we will consider the dynamics of microbubbles driven by the Westervelt pressure input. More precisely, we will simulate the Westervelt--Rayleigh--Plesset system with the zero acoustic source term (that is, with $\xi=0$). The two numerical studies will allow us to distinguish the source of nonlinear contributions and comparatively observe the effects of modeling ultrasound on the microbubble dynamics.\footnote{The program code for all simulations in this section is available as an ancillary file from the arXiv page of this paper \hyperlink{https://arxiv.org/abs/2408.06108}{(arXiv:2408.06108)}.} \\ % (arXiv:...).
\setlength{\belowcaptionskip}{-10pt} 
\indent To ensure accurate and meaningful numerical simulations, we employ parameter values typical for ultrasound contrast imaging below, see, e.g., \cite{cobbold2007foundations}. As one of the most widely used ultrasound contrast agents is SonoVue™ (Bracco SpA), see, e.g., \cite{versluis2020ultrasound}, we choose the microbubble parameter values relevant for it.}
\subsection{Numerical framework for solving Rayleigh--Plesset equations}
We assume $\rho_0 \approx \rho_{l0}$ and focus on two concrete equations described in Section \ref{sec: bubble models} modeling coated and non-coated microbubbles. 
%Our study here centers on the dynamics of single non-coated bubbles using the Rayleigh–Plesset type equation given \eqref{RP_4} and \eqref{RP_shell_1}.
%To ensure accurate and meaningful numerical simulations, we use parameter values typical for ultrasound contrast imaging wherever possible. One of the most widely used ultrasound contrast agents is SonoVue™ (Bracco SpA), see, e.g., \cite{versluis2020ultrasound}. 
{We employ parameter values relevant for SonoVue™, which has a number concentration of $1 \cdot 10^8$ to $5 \cdot 10^8$ microbubbles/$\si{\milli \liter}$ and the mean diameter of the microbubbles is $2 \, \si{\micro \m}$.} Fixed parameter values are listed in the following table.

\begin{table}[h]
	\begin{center}
		\begin{tabular}{l l l | l l l}
			\hline
			$R_0$ & initial radius  & $2 \, \si{\micro \m}$ & $\kappa$ & adiabatic exponent & $1.4$ \\	
			$\kappa_s$ & shell viscosity & $ 2 \cdot 10^{-6} \, \si{\kilogram / \s}$ &
			$\rho_0$ & mixture mean mass & $1000 \, \si{\kilogram / \m^3}$ \\
			$\mu$ & shear viscosity & $8.9 \, \si{\milli \pascal \, \s}$ &
			$\sigma$ & surface tension & $72.8 \, \si{\milli \newton / \m}$ \\
			$c$ & speed of sound & $1500\,\si{\m / \s}$ & $\chi$ & shell elasticity  & $2 \, \si{\newton / \m}$\\
			$p_v$ & vapor pressure & $2330 \, \si{\pascal}$ &
			$p_{\text{stat}}$ & ambient pressure &  $100 \, \si{\kilo \pascal}$ \\
			\hline
		\end{tabular}
	\end{center}
	\caption{Overview of the parameter values used in the simulations.}
	\vspace{-0.0cm}
\end{table}
%	\vspace{-1.3cm}
As we are dealing with second-order ODEs that are nonlinear and singular, the numerical simulations can be instable, highly fluctuating or produce negative values for the bubble radius. These problems especially arise when the radius becomes very small and consequently the rate of change of the radius is extremely large. For these reasons, we use the fourth-order Runge Kutta scheme combined with the adaptive approach proposed by~\cite{LEE2020593}. The time step in the discrete ODEs is dependent on the size of the bubble radius and chosen according to
\begin{equation} \label{eq: adaptive time stepping}
	\Delta t (t_i) = [R(t_i)]^{\lambda}
\end{equation} 
with a time index $\lambda$. This means that as the value of $\lambda$ increases, the time steps become smaller, leading to longer numerical simulations but higher accuracy. By setting $\lambda=1.75$, we achieve a good balance between computational effort and accuracy, resulting in time steps ranging from $10^{-12} \si{\second}$ to $10^{-9} \si{\second}$. We note that these observations are empirical, as a rigorous study of the used time-stepping method does not appear to be available in the literature at present.

\subsection{Single-microbubble dynamics: Coated vs.\ non-coated} \label{sec: single bubble dynamics}
We first discuss the simulation of the Rayleigh--Plesset equations,  focusing on the dynamics of single microbubble under the influence of a sinusoidal driving pressure. We consider a driving pressure function $p(t) = A \sin( 2 \pi f t) $
where, in this section, $f$ is the driving frequency and $A$ is the amplitude. We present selected settings for various values of $A$ and $f$ in the following.  
Figure \ref{fig: overview equ complex} shows  $R=R(t)$ using coated-bubble equation given in \eqref{RP_shell_1} with varying amplitudes $A = 1, \, 10 \, \si{\mega \pascal}$ and driving frequencies $f = 0.2, \, 0.5 \, \si{\mega \hertz}$.  

\begin{figure}[h]
	\vspace{-0.3cm}
	\centering
	\includegraphics[width=0.85\textwidth]{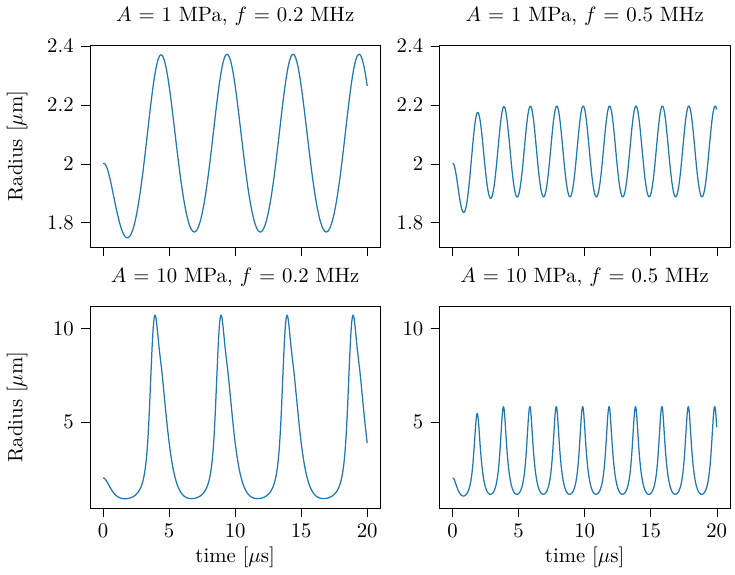}
    %\scalebox{0.85}{\input{equ3_3_2.tex}}
	\vspace{-0.3cm}
	\caption{Sensitivity of a coated microbubble to driving amplitude and frequency.}
	\label{fig: overview equ complex}
	\vspace{-0.5cm}
\end{figure}
The results in Figure \ref{fig: overview equ complex} clearly illustrate the influence of the driving pressure on the bubbles' radii. At low forcing levels, the microbubbles exhibit almost sinusoidal oscillations with relatively small amplitude. As the forcing pressure increases, the effects of nonlinearity become significantly more pronounced. Specifically, for a fixed frequency, we observe that the curve representing the radius steepens, and the amplitude of the radius increases as the pressure amplitude rises. Conversely, when comparing different frequencies but maintaining the same pressure amplitude, lower frequencies reveal more nonlinear effects, with the amplitude of the radius being higher. This indicates that both the magnitude of the driving pressure and the frequency play crucial roles in dictating the bubble’s dynamic behavior.  

 In diagnostic ultrasonic imaging, the typical frequency range  is $1$ to $10 \, \si{\mega \hertz}$. For the amplitude of the driving pressure, realistic values are between $10$ and $15 \, \si{\mega \pascal}$. Figure \ref{fig: radius equ complex FFT} shows $R=R(t)$ for coated bubbles and the corresponding fast Fourier transform (FFT) - spectra for the amplitude $A = 15 \, \si{\mega \pascal}$ and different driving frequencies $f = 0.5, \, 1, \, 5 \, \si{\mega \hertz}$. To facilitate a thorough comparison, we examine 10 cycles of the curves in time domain and the corresponding frequency spectra. The results indicate that nonlinear effects decrease as the frequency increases. For the smallest frequency of $0.5 \, \si{\mega \hertz}$ the radius curve is very steep and the bubble reaches approximately six times its initial size. In frequency domain the harmonics at multiples of the fundamental frequency are clearly visible and far more intense at a driving frequency of $0.5 \, \si{\mega \hertz}$ compared to higher frequencies of $1$ and $5 \, \si{\mega \hertz}$. This is probably at least partially due to stronger attenuation at higher frequencies, as also visible in the reduced amplitudes (around $10 \, \si{\micro\meter}$ at $0.5 \, \si{\mega\hertz}$ and only around $2 \, \si{\micro\meter}$ at $5 \, \si{\mega\hertz}$). 
 
 \begin{figure}[h]
 	\centering
 	\vspace{-0.3cm}
 	%	\vspace{-0.2cm}
 	\includegraphics[width=0.85\textwidth]{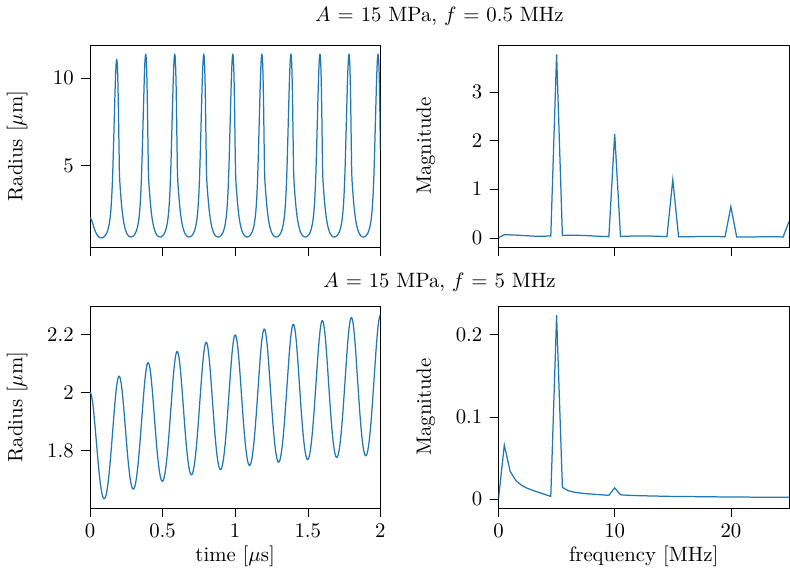}
 	%\scalebox{0.85}{\input{freq051.tex}}
 	\vspace{-0.3cm}
 	%\scalebox{0.9}{	\input{freq5.tex}}
 	\caption{Behavior of a coated microbubble for high frequencies: Radius-time $R(t)$ curves and the corresponding FFT-spectra. }
 	\label{fig: radius equ complex FFT}
 	\vspace{-0.5cm}
 \end{figure}

Next, we want to compare the radius plots resulting from the Rayleigh--Plesset type equation for coated bubbles, given by equation \eqref{RP_shell_1}, with those from equation \eqref{RP_4} for non-coated bubbles, where the shell terms are neglected. 
%The presence of a shell adds stability to the bubbles by reducing nonlinearities and keeping the radius amplitude smaller. 
\begin{figure}[h]
	\centering
	\includegraphics[width=0.85\textwidth]{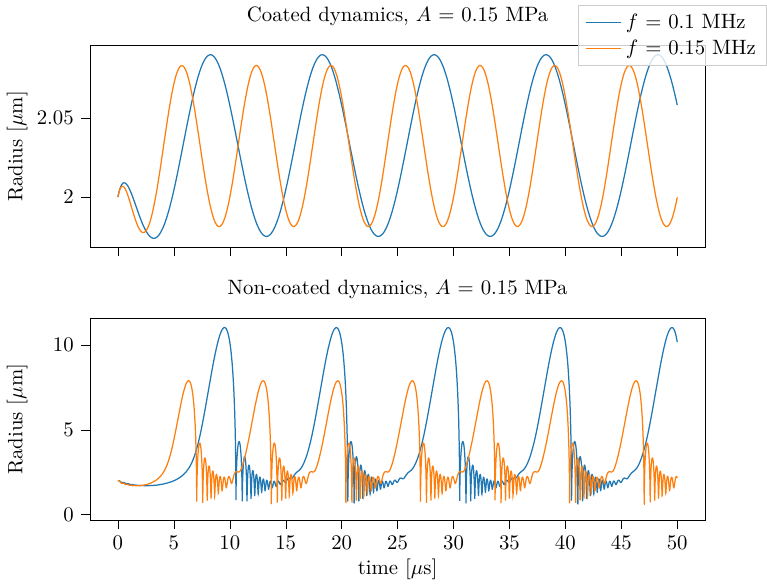}
	%\scalebox{0.85}{\input{equ1_2.tex}}
	\vspace{-0.3cm}
	\caption{Comparison of the dynamics of a coated and non-coated microbubble.}
	\label{fig: comparison equ easy and complex}
	\vspace{-0.5cm}
\end{figure}
For an amplitude of $A = 0.15 \, \si{\mega \pascal}$ and driving frequencies $f = 0.1, \, 0.15 \, \si{\mega \hertz}$,  the first plot at the top in Figure \ref{fig: comparison equ easy and complex} shows the numerical solutions $R=R(t)$ of the ODE \eqref{RP_shell_1} and the second one of equation \eqref{RP_4}. The simulations that take the shell terms into account show a sinusoidal behavior with an amplitude close to the initial radius. The non-coated bubble expands approximately four to six times its initial radius, drops extremely quickly and afterbounces roughly with its eigenfrequency. This clearly illustrates that the shell protects the bubbles and dampens the fluctuations in the bubble radius. \\
\indent To complete our study of single-microbbuble dynamics, we also consider the Rayleigh–Plesset type equation for non-coated bubbles on its own. Given the inherent instability of this model, high amplitudes and frequencies -- typical in ultrasound contrast imaging -- can cause numerical instabilities. When these parameters are increased to more realistic levels the simulation of equation \eqref{RP_4} becomes unstable, resulting in negative radius values. Therefore, we consider driving amplitudes of $A = 0.05, \, 0.15 \, \si{\mega \pascal}$ with a fixed frequency of $f = 0.15 \, \si{\mega \hertz}$. 
\begin{figure}[h]
	\centering
	\includegraphics[width=0.85\textwidth]{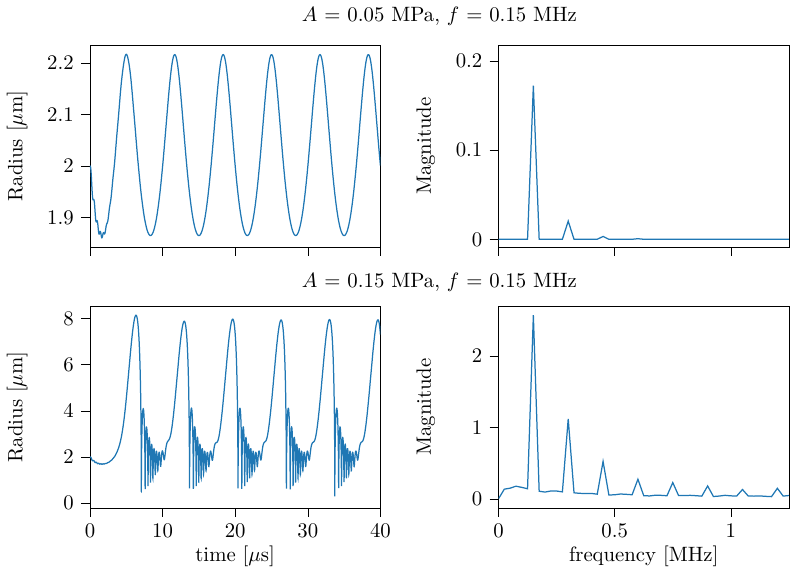}
	%\scalebox{0.85}{\input{freq2_015_0.tex}}
	\vspace{-0.3cm}
	%\input{freq2_015_1.tex}
	%\scalebox{0.9}{	\input{freq2_015_5.tex}}
	\caption{Behavior of a non-coated microbubble: Radius-time $R(t)$ curves and the corresponding FFT-spectra.}
	\label{fig: different amplitudes equ easy}
	\vspace{-0.5cm}
\end{figure}
The resulting Radius-time curves $R(t)$ are presented in Figure \ref{fig: different amplitudes equ easy}, illustrating different regimes of bubble oscillation behavior: from sinusoidal patterns to large bubble growth followed by a steep drop and rebounds. As the amplitude increases, the time-domain curves exhibit small oscillations, which can be observed in the frequency domain as higher harmonics rise, indicating stronger nonlinear effects. Note that in Figure \ref{fig: radius equ complex FFT} we have kept the driving amplitude constant while varying the frequencies, whereas in Figure \ref{fig: different amplitudes equ easy}, we have varied the amplitudes with a fixed frequency. This comparison reveals that nonlinear effects in microbubble dynamics intensify as the driving frequency decreases and the driving amplitude increases.

%Another factor that influences the stability is the polytropic exponent. We observed that it has a smoothing and dampening effect on the curves of the radius for equation \eqref{RP_4}, see Figure \ref{fig: influence_exponent}. In equation \eqref{RP_shell_1} the other effects are more dominant and different polytropic exponents give almost identical curves. 
%\begin{figure}%[h]
%	\centering
%	\input{simple_equ_diff_gamma.tex}
%	\caption{Radius-time $R(t)$ curves for the amplitude $A = 0.1 \, \si{\mega \pascal}$, driving frequency $f = 0.15 \, \si{\mega \hertz}$ and different polytropic exponents $\kappa = 1, \, 1.5, \, 2$ of equation \eqref{RP_4}. }
%	\label{fig: influence_exponent}
%\end{figure}
%\newpage
\subsection{The influence of ultrasound on microbubbles}\label{sec:numwest}
We next present numerical results for the Westervelt--Rayleigh--Plesset model. To somewhat simplify the computational complexity, we set the right-hand source term in the Westervelt equation to zero (that is, we take $\xi=0$). {This setting allows solving the wave-ODE system sequentially; that is, we can first solve the PDE and then use it as an input for the ODE. This way, we can keep the time step used in the discretization of the wave equation larger (by several orders of magnitude) compared to the Rayleigh--Plesset equation.}\\
% By incorporating pressure inputs from Westervelt's equation we enhance the realism of our simulations and aim to gain insights into the complex interplay between the driving pressure and bubble dynamics, ultimately contributing to the optimization of ultrasound-based technologies. 
\indent {When it comes to the rigorous numerical analysis of the Westervelt equation, its finite element semi-discretization in space is best understood theoretically. \emph{A priori} analysis of the semi-discrete equation using conforming finite elements under homogeneous Dirichlet conditions can be found in~\cite{nikolic2019priori}. It is known that optimal error of convergence in the $L^2(\Omega)$-based norms is achieved for sufficiently small $h$ and the exact solution; we refer to~\cite[Theorem 6.1]{nikolic2019priori} for details.} To have a more realistic computational setting, we employ Neumann boundary conditions in the simulations and zero pressure data.  To achieve focusing, we construct a rectangular domain with one curved side, where the waves are excited such that they focus at the center of the domain.  The Neumann conditions are given by $\frac{\partial p}{\partial n}= A_p \sin(2 \pi f_p t)$ on the curved part of the boundary of the computational domain, and set to zero elsewhere. The amplitude of the acoustic boundary excitation is taken to be $A_p = 0.1$ MPa/m and the frequency $f_p =15$ kHz. All simulations of Westervelt's equation are performed using FEniCSx v0.7.2, see, e.g., \cite{baratta_2023}, with Gmsh~\cite{gmsh_2009} used for meshing. The simulations are based on using continuous linear finite elements for spatial discretization and a predictor-corrector Newmark scheme for time integration, following the algorithm in~\cite[Ch. 5]{kaltenbacherMathematicsNonlinearAcoustics2015}. The integration parameters in the Newmark scheme are set to $(\gamma, \beta) = (0.7, 0.4)$ and the time step is taken to be $3 \cdot 10^{-6} \si{\second}$. Note that this time step differs in size by {six orders of magnitude} compared to the time step used for the Rayleigh--Plesset ODEs. %The significant difference in time step sizes introduces challenges, particularly in ensuring stability and accuracy across different time scales. 

%We  begin by presenting pressure waves generated from Westervelt's equation with strong acoustic damping using Neumann boundary conditions and zero initial data, within a focused domain

\begin{figure}[h]
	\centering
	\includegraphics[scale=0.35]{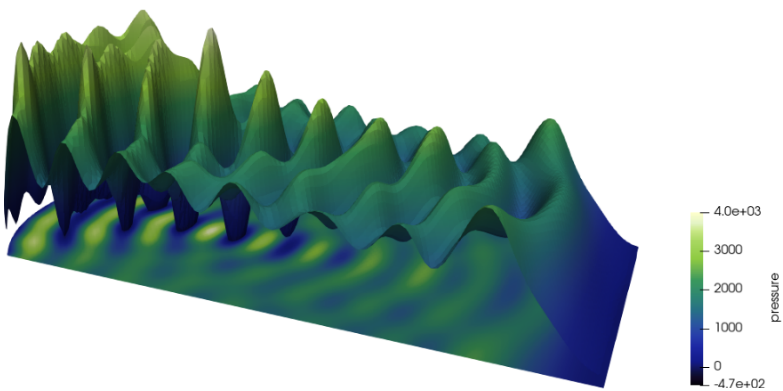}
	\vspace{-0.2cm}
	\caption{Propagation and self focusing of nonlinear pressure waves.}
	\label{fig: west}
	\vspace{-0.5cm}
\end{figure}
Figure \ref{fig: west} shows the pressure field generated from Westervelt's equation. The focal region, the area with the highest peak of the pressure waves, is clearly visible and shows the point of maximum pressure intensity. Additionally, the figure highlights nonlinear effects: the positive peak pressure surpasses the negative values, indicating an asymmetry in the pressure distribution and the side profile of the waveform shows noticeable steepening.

\begin{figure}[h]
	\centering
	\includegraphics[width=0.9\textwidth]{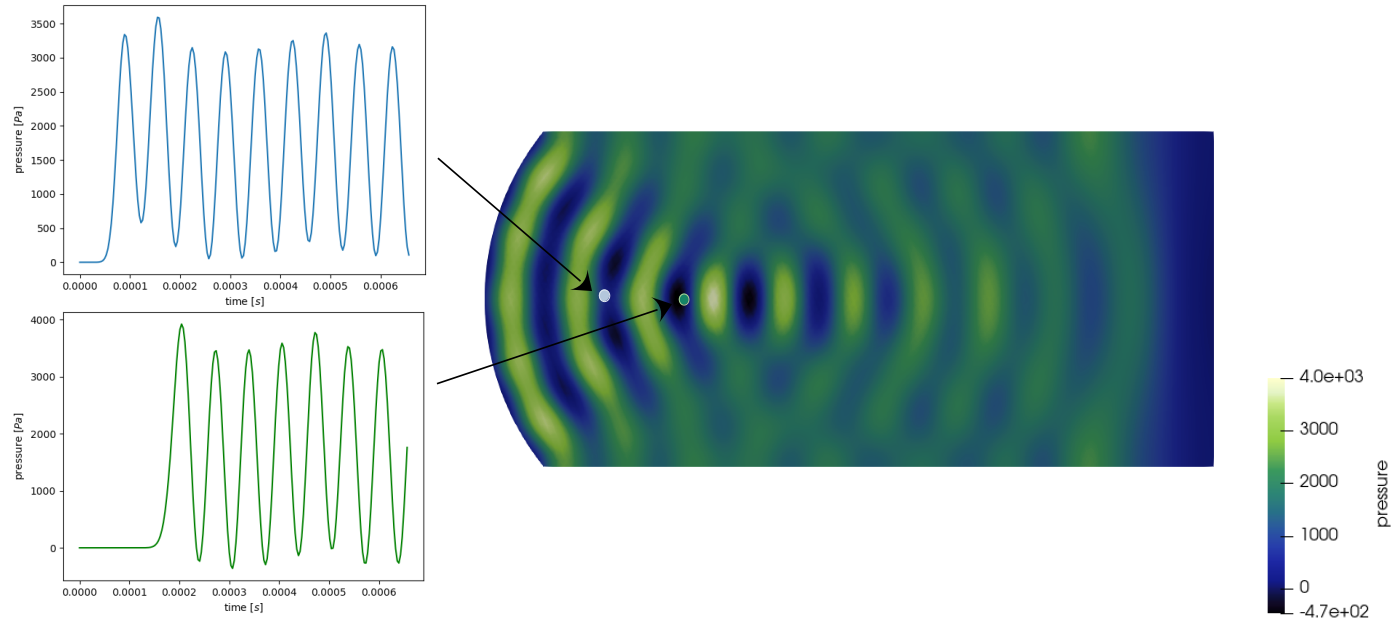}
	\vspace{-0.2cm}
	\caption{Westervelt input for RPE: Acoustic pressure waves over time at two different spatial points obtained using Westervelt's equation.}
	\label{fig: west_overview}
	\vspace{-0.5cm}
\end{figure}
Next, we select two pressure-time curves obtained along the central axis to serve as inputs for the Rayleigh--Plesset-type equations; see~Figure \ref{fig: west_overview}. 
The first curve, depicted in blue, is representative of the pressure close to the excitation source. In contrast, the second curve, shown in green, corresponds to the focal region where the pressure is more intense. As a result, the green curve demonstrates more prominent nonlinear effects compared to the blue one. 
%We chose to examine these two curves for the problem stated in 
% eqref{ibvp_WestRP} in order to understand how different regions within the acoustic field influence %microbubble behavior. The results of this analysis are illustrated in 
Figure~\ref{fig: west sigma(R)} shows that the amplitude of the green radius-time curve is significantly higher compared to the blue curve. This increase in amplitude is a direct result of the greater intensity of the pressure input represented by the green curve. The corresponding frequency spectra further highlight this effect, revealing that while higher harmonics are not prominently visible, subharmonics seem to appear.\\
\indent When comparing these findings to the results presented in Section \ref{sec: single bubble dynamics}, where a sinusoidal driving pressure was used, it becomes apparent that {the additional nonlinear effects that arise from the pressure itself change the bubble dynamics significantly.} The radius-time curves shown in Figure \ref{fig: west sigma(R)} exhibit greater symmetry with respect to the $x$-axis and demonstrate expanding bubble behavior. %As expected, a sinusoidal driving pressure leads to enhanced stability and generates periodic radius-time curves.
 This comparison underscores the differences in microbubble dynamics influenced by varying {(Westervelt-based)} pressure inputs.

\begin{figure}[h]
	\centering
	\includegraphics[width=0.9\textwidth]{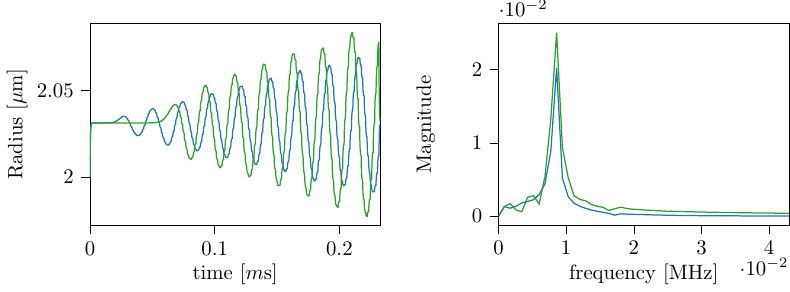}
	%\scalebox{0.85}{\input{west_ode3.tex}}
	\vspace{-0.2cm}
	\caption{Coated microbubbles with Westervelt input: Radius-time $R(t)$ curves and the corresponding FFT-spectra at two different positions in the domain; see also Fig. \ref{fig: west_overview}.}
	\label{fig: west sigma(R)}
	\vspace{-0.5cm}
\end{figure}

We conclude our investigations by examining the problem involving Westervelt's equation in conjunction with the non-coated microbubble dynamics (modeled by \eqref{RP_4}).  Figure \ref{fig: west sigma} reveals that this ODE exhibits greater instability and stronger nonlinear effects. In particular, the frequency domain analysis shows an increase in higher harmonics, a trend that becomes more pronounced with higher pressure input intensities. Additionally, the gradients of the radius-time curves are steeper, the curves display reduced symmetry, and the amplitudes are greater compared to those shown in Figure \ref{fig: west sigma(R)}.

\begin{figure}[h]
	\centering
\includegraphics[width=0.9\textwidth]{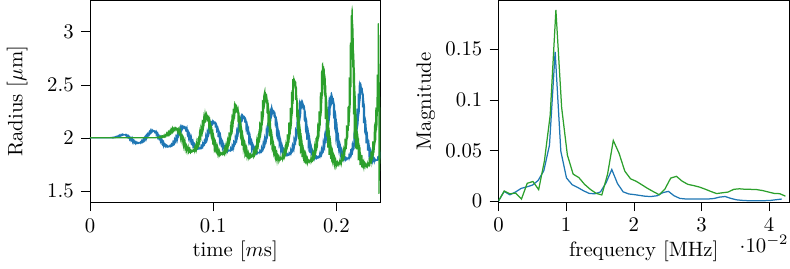}
	%\scalebox{0.85}{\input{west_ode2.tex}}
	\vspace{-0.2cm}
	\caption{Non-coated bubbles with Westervelt input: Radius-time $R(t)$ curves and the corresponding FFT-spectra at two different positions in the domain; see also Fig. \ref{fig: west_overview}.}
	\label{fig: west sigma}
	\vspace{-0.5cm}
\end{figure}

	\section{Outlook}
	In this work, we have explored the mathematical modeling and analysis of ultrasound contrast imaging, with a focus on the nonlinear effects in ultrasound propagation and interactions between microbubbles and pressure waves. 	Adapting the mathematical analysis to the setting of Neumann or absorbing acoustic boundary conditions in future reseach is important for modeling more realistic scenarios, where one should avoid non-physical reflections of sound waves. Furthermore, extending the analysis to address the refined wave-ODE models derived in Section~\ref{Sec:Modeling}, such as \eqref{eq: merged reduced second}, would provide deeper insights into the complex interactions between acoustic waves and microbubble dynamics. From a numerical perspective, addressing the fully coupled problem \eqref{ibvp_fractionalWestRP_general} with $\xi \neq 0$, where both wave propagation and microbubble dynamics are solved simultaneously, would be a significant step forward in capturing the complete range of physical behaviors. 
	This setting introduces additional computational complexity, particularly due to the need for much finer time steps in discretizing the ODE compared to the wave equation, which already need to be relatively small in the nonlinear acoustic setting. Techniques such as multirate time stepping and operator splitting methods could be explored here to manage differing time scales in the wave and ODE dynamics. Furthermore, rigorous numerical analysis of the schemes used in the present work and those developed in the future would provide a deeper understanding of the simulated problems.
	\noindent \subsection*{Acknowledgments} We are grateful to Prof.\ Barbara Kaltenbacher (University of Klagenfurt) for her valuable comments. We thank Prof.\ Martin Verweij (Delft University of Technology) for pointing out references~\cite{matalliotakis2023computation} and~\cite{matalliotakis2024impact} to us. This research was funded in part by the Austrian Science Fund (FWF) [10.55776/DOC78].\\
	\indent  For open-access purposes, the authors have applied a CC BY public copyright license to any author-accepted manuscript version arising from this submission. 
	%%%%%%%%%%%%%%%%%%%%%%%%%%%%%%%
	\begin{appendices} 
	\section{Appendix} \label{Appendix}
	\subsection{Lipschitz continuity results} \label{Appendix: Lipschitz}
	We include here the derivation of estimate \eqref{Lipschitz continuity p} for the pressure field.
	\begin{lemma}\label{Lemma: est diff p}
		Under the assumptions of Proposition~\ref{Prop:WellP_West}, let $\pone$ and $\ptwo$ be the solutions of \eqref{ibvp_fractionalWestRP_general} with the right-hand sides $\fone, \ftwo \in L^2(0,T; \Ltwo)$, respectively. Then
		\[
				\| \pone - \ptwo\|_{C(\Htwo)} \lesssim \|\fone - \ftwo\|_{L^2(\Ltwo)}.
		\]
	\end{lemma}
	\begin{proof}
		Thanks to Proposition~\ref{Prop:WellP_West}, $\pone$, $\ptwo \in \Xp$, and
		\begin{equation} \label{bound pone ptwo strong damping}
			\|\pone\|_{\Xp}, \ \|\ptwo\|_{\Xp} \leq \mathcal{C}
		\end{equation}
		for some $\mathcal{C}>0$.  The difference $\op= \pone- \ptwo$ solves
		\begin{equation}
			\begin{aligned}
				(1+2k(x)\pone) \op_{tt}- c^2 \Delta \op - b \Delta  \op_t = \fone -\ftwo - 2k(x) \bigl(\op \ptwo_{tt} + \op_t(\pone+\ptwo)\bigr).
			\end{aligned}
		\end{equation}
		with homogeneous boundary and initial data.	The claim follows by testing this difference equation with $\op_{tt} - \Delta \op \in \LtwoTLtwo$, integrating over $\Omega$ and $(0,t)$, and performing integration by parts in time and space, analogously to the arguments in the proof of contractivity in~\cite[Proposition 1]{kaltenbacher_rundell2021}. In this way, we obtain
		\begin{equation} \label{est op}
			\begin{aligned}
				&\begin{multlined}[t]
					\intt \|\sqrt{1+2k \pone}\op_{tt}\|^2_{\Ltwo}\ds+\negthickspace c^2 \intt \|\Delta \op\|^2_{\Ltwo}\ds+\frac{b}{2}\|\Delta \op(t)\|^2_{\Ltwo}+\frac{b}{2}\|\nabla \op_t\|^2_{\Ltwo}
				\end{multlined}\\
				\leq& \inttO((1+2k\pone) \op_{tt}) \Delta \op \dxs +	c^2\int_0^t \intO \Delta \op\, \op_{tt} \dxs\\
				&+\negthickspace\int_0^t \negthickspace \intO (\fone-\ftwo) (\op_{tt} \negthickspace -\Delta \op) \dxs-2 \negthickspace \int_0^t \negthickspace \intO  \negthickspace k(x) \bigl(\op \ptwo_{tt} \negthickspace+ \op_t(\pone \negthickspace +\ptwo)\bigr)(\op_{tt}\negthickspace-\Delta \op) \dxs.
			\end{aligned}
		\end{equation}
		Indeed, on account of \eqref{bound pone ptwo strong damping}, we have
		\begin{equation}
			\begin{aligned}
				\int_0^t \negthickspace \intO \negthickspace (	(1\negthickspace+\negthickspace2k\pone) \op_{tt}) \Delta \op \dxs \leq&\, \|1\negthickspace+\negthickspace2k \pone \negthinspace \|_{\LinfLinf}\|\op_{tt}\|_{L^2(0,t; \Ltwo)} \|\Delta \op\|_{L^2(0,t; \Ltwo)} \\
				\lesssim&\, \eps \|\op_{tt}\|^2_{L^2(0,t; \Ltwo)} +\|\Delta \op\|^2_{L^2(0,t; \Ltwo)}
			\end{aligned}
		\end{equation}
		for any $\varepsilon>0$. Furthermore,
		\begin{equation}
			\begin{aligned}
				\int_0^t \intO (\fone-\ftwo) (\op_{tt}-\Delta \op) \dxs \leq&\, \begin{multlined}[t] \frac{1}{2\eps} \|\fone-\ftwo\|^2_{L^2(0,t; \Ltwo)}+\eps \|\op_{tt}\|^2_{L^2(0,t; \Ltwo)}\\+\eps \|\Delta \op\|^2_{L^2(0,t; \Ltwo)},
				\end{multlined}
			\end{aligned}
		\end{equation}
		and the other terms can be treated analogously. By choosing a sufficiently small $\varepsilon>0$ and employing these estimates in \eqref{est op}, the $\op_{tt}$ terms on the right-hand sides can be absorbed by the left-hand side, whereas the $\Delta \op$ terms can be subsequently handled using Gr\"onwall's inequality. In this manner, we obtain the claimed estimate.
	\end{proof}
	Finally, we prove here that assumptions \eqref{assumption1 h_0} and \eqref{assumption2 h_0} hold for the right-hand side $\hzero$ in the Rayleigh--Plesset equation given by 
	\begin{equation} \tag{\ref{def hzero}}
		\begin{aligned}
			h_0(R, R_t)=\begin{multlined}[t]   \pb -   \frac{4 \mu}{R}\Rt -\frac{2\sigma(R)}{R} +\ppgn \left(\frac{\Rn}{R} \right)^{3 \kappa}{\left(1- 3 \kappa_0 \frac{\Rt}{c}\right)}  -4 \kappa_s \dfrac{\Rt}{R^2},
			\end{multlined}
		\end{aligned}
	\end{equation}  
	where $\sigma$ is a constant or $\sigma(R)=\chi(\tfrac{R^2}{R_0^2}-1)$ and $\kappa_s$, $\kappa_0 \in \R$, provided $R \in \BR$. 
	\begin{lemma} \label{Lemma: est diff hzero}
		Under the assumptions of Theorem~\ref{Thm:LocalWellp_WestRP}, let $\Rsone$, $\Rstwo \in \BR$. Let $h_0$ be given by \eqref{def hzero} with $\pb$, $\ppgn \in C([0,T]; \Linf)$, $\kappa \geq 1$, $\kappa_0$, $
		\kappa_s$, $\mu \in \R$, and $\Req>0$. Assume $\sigma$ is either a real constant or $\sigma(R)=\chi(\tfrac{R^2}{R_0^2}-1)$ with $
		\chi \in \R$. Then $\hzero$ satisfies \eqref{assumption1 h_0} and \eqref{assumption2 h_0}.
	\end{lemma}
	\begin{proof}
		We provide the proof of \eqref{assumption2 h_0}. We rewrite the difference of $h_0$ terms as
		\begin{equation}\label{diff h terms}
			\begin{aligned}
				&\hzero(\Rsone, \Rsone_t)-\hzero(\Rstwo, \Rstwo_t) \\
				=& - \tfrac{1}{\Rsone \Rstwo}\pb(\Rsone-\Rstwo) -4 \mu   \frac{\tR^*_t \Rstwo+ \Rstwo_t\tR^*}{\Rsone \Rstwo}\\
				&-2\frac{\sigma(\Rsone)\tR^*+(\sigma(\Rsone)-\sigma(\Rstwo))\Rstwo)}{\Rsone \Rstwo} +\ppgn \negthinspace\left(\tfrac{\Rn}{\Rsone} \right)^{3 \kappa }\left(1\negthinspace-3\kappa_0 \tfrac{\Rsone}{c}\right)\\
				&-\ppgn \negthinspace\left(\tfrac{\Rn}{\Rstwo} \right)^{3 \kappa}\negthickspace\left(1\negthinspace-3\kappa_0 \tfrac{\Rstwo}{c}\right) \negthickspace-4 \kappa_s \tfrac{1}{(\Rsone)^2}\tRs \negthickspace-\negthinspace4\kappa_s \Rstwo_t \negthickspace\frac{\tRs}{(\Rsone \Rstwo)^2}.
			\end{aligned}
		\end{equation}
		Note that since $\Rsone \in \BR$, we have
		\begin{equation}
			\|\sigma(\Rsone)\|_{\CLinf} \lesssim 1+ \|\Rsone\|^2_{\CLinf} \lesssim 1+m^2.
		\end{equation}
		Using
		\begin{equation}
			\begin{aligned}
				\left(\negthinspace\frac{1}{\Rsone}\negthinspace \right)^{3 \kappa} \negthickspace- \left(\negthinspace\frac{1}{\Rstwo} \negthinspace\right)^{3 \kappa} \negthickspace=\left(\frac{1}{\Rsone \Rstwo}\right)^{3 \kappa} \negthinspace\tR^\kappa  \left((\Rsone)^2 \negthickspace+ \Rsone \Rstwo \negthickspace + (\Rstwo)^2\right)^\kappa
			\end{aligned}
		\end{equation}
		as well as
		\[
		\|\sigma(\Rsone)-\sigma(\Rstwo)\|_{\CLinf} \negthickspace \leq \negthinspace \frac{|\chi|}{\Rn^2}\|\tR\|_{\CLinf}(\negthinspace\|\Rsone\|_{\CLinf}+\|\Rstwo\|_{\CLinf} \negthinspace)
		\]
		and the fact that the radii are uniformly bounded:
		\begin{equation}
			\begin{aligned}
				&0<\ulR \leq	\Rsone,\, \Rstwo \leq \olR, \quad \|\Rsone\|_{\ConeLinf}+\|\Rstwo\|_{\ConeLinf} \lesssim m, 
			\end{aligned}
		\end{equation}
		from \eqref{diff h terms} we obtain
		\[
		\begin{aligned}
			&\|\hzero(\Rsone, \Rsone_t)-\hzero(\Rstwo, \Rstwo_t)\|_{\CLinf} \\
			\lesssim&\,  \|\Rsone-\Rstwo\|_{\CLinf}+\|\Rsone_t-\Rstwo_t\|_{\CLinf},
		\end{aligned}
		\]
		as claimed.
	\end{proof}

	\subsection{Extension of the modeling framework to time-fractional acoustic attenuation} \label{Appendix: Fractional}
	We can modify the Westervelt--Rayleigh--Plesset system given in \eqref{ibvp_fractionalWestRP_general} in a straightforward manner to incorporate time-fractional acoustic attenuation instead. Time-fractional damping is known to match well the observed acoustic attenuation in complex media, such as soft tissue; see, e.g.,~\cite{prieur2011nonlinear}. In the following, we discuss how this model can be derived and how the analysis can be adapted accordingly. We have also performed simulations similar to those described in Section \ref{sec:numwest} and observed comparable effects; therefore, we do not present them here.
	
	\subsubsection{Derivation of the system} 
	We next adapt the derived models by introducing fractional derivatives following \cite{prieur2011nonlinear}. Therefore, instead of \eqref{eq: conservation of momentum} we make use of the fractional integral generalization of the conservation of momentum given by
	\begin{equation} \label{eq: frac momentum equ}
		\rho \left( \bfu_t \negthinspace + \negthinspace \left( \bfu \cdot \nabla\right)  \bfu \right) \negthinspace + \negthinspace \nabla p = \mu \tau_1^{\alpha_1 -1} I^{1-\alpha_1}\Delta \bfu \negthinspace + \negthinspace \left( \frac{\mu}{3} \negthinspace + \negthinspace \mu_b \right) \tau_1^{\alpha_1 -1} I^{1-\alpha_1} \left( \nabla \left( \nabla \cdot \bfu \right)\right) 
	\end{equation}
	for $0 < \alpha_1 \leq 1$, where $\tau_1$ is a time constant characteristic of the creep time of the medium. Further, $I^{1-\alpha_1}(\cdot)$ represents the fractional integral of order $1-\alpha_1$ that is defined as follows:
	\begin{equation}
		I^{y}[f(t)] = \frac{1}{\Gamma(y)} \int_0^t (t-s)^{y-1} f(s) \ds \qquad \text{ for } 0<y.
	\end{equation}
	The fractional integral appearing in \eqref{eq: frac momentum equ} originates from using a fractional Kelvin--Voigt model for the stress tensor; we refer to~\cite{prieur2011nonlinear} for details. We also make use of the fractional form of the generalized state equation given by
	\begin{equation} \label{eq: frac state equation}
		\rho_l' = \frac{p'}{c^2} - \frac{1}{\rho_{l0} c^4 } \frac{B}{2A} p'^2- \frac{\gamma \tau_2^{\alpha_2-2}}{\rho_{l0} c^4 } \left( \frac{1}{c_v} - \frac{1}{c_p} \right) D_t^{\alpha_2-1} p'
	\end{equation}
	for $1 <\alpha_2 \leq 2$, where $ D^{\alpha_2}_t(\cdot)$ stands for the Caputo--Djrbashian fractional time derivative of order $\alpha_2$ that results from having a fractional entropy equation in the medium; see~\cite{prieur2011nonlinear}. It is defined as follows:
	\begin{equation} \label{def Caputo}
		D_t^{y} f(t)= \frac{1}{\Gamma(1-r)} \int_0^t (t-s)^{-r} D_t^n f(s) \ds,
	\end{equation}
	where $0 \leq n-1 < y<n$, $r= y - n+1$, $n \in \N$; see~\cite[Sec.\ 1]{kubica2020time}. \\
	\indent	Equations \eqref{eq: bfmass}, \eqref{eq: conservation of mass}, \eqref{eq: conservation of bubbles}, and \eqref{eq: continuum density l} together with the fractional equations \eqref{eq: frac momentum equ} and \eqref{eq: frac state equation} and an ODE for the bubble radius are a set of seven equations for the seven fluctuations as before. We can then follow the derivation analogously to above by  making use of Lighthill's scheme and the substitution corollary combining these equations. \\
	\indent With the approximations above and 
	\begin{align}
		I^{1-\alpha_1}\left(  \nabla \left( \nabla \cdot \bfu\right)\right)  &= I^{1-\alpha_1} \bigl( \nabla \cdot \nabla \bfu\bigr)  = I^{1-\alpha_1}  \bigl( - \tfrac{1}{n_0} \nabla n_t' \bigr) =  - \tfrac{1}{n_0} D_t^{\alpha_1} \nabla n',
	\end{align}
	where we have made use of the property
	\begin{equation}
		D^{\alpha_1}_t [I^{\alpha_2}](\cdot) = \begin{cases} D^{\alpha_1-\alpha_2}_t (\cdot) \quad & \text{ if } 0<\alpha_2<\alpha_1, \\
			I^{\alpha_2-\alpha_1}(\cdot)  & \text{ if } 0<\alpha_1<\alpha_2, \end{cases}
	\end{equation}
	we arrive at 
	\begin{equation} \label{eq: frac momentum3}
		\rho_0 \bfu_t \negthinspace + \nabla p' \negthinspace  =  - \tfrac{1}{1-\bmass_0}\bmass' \nabla p' \negthinspace +\tfrac{1-\bmass_0}{2 \rho_0 c^2} \nabla p'^2 - \tfrac{\rho_0}{2} \nabla \left( \bfu \cdot \bfu \right) -	\tau_1^{\alpha_1 -1}  \tfrac{\left( \frac{4}{3} \mu + \mu_b \right)}{n_0} D_t^{\alpha_1} \nabla n'.
	\end{equation}
	Analogously to deriving equation \eqref{eq: merged3}, we obtain 
	\begin{equation} \label{eq: merged3 frac}
		\begin{aligned}
			\Delta p' -& \frac{1-\bmass_0}{c^2} p'_{tt} +  \frac{\rho_0}{1-\bmass_0} \bmass'_{tt} +  \frac{\tilde{\beta}}{\rho_0 c^4}  ( p'^2)_{tt} +\frac{\tau^{\alpha-1}\tilde{ b}}{c^2} D^{\alpha}_t \Delta p' \\
			=& \begin{multlined}[t]  -  \frac{1}{1-\bmass_0} \bmass' \Delta p' -  \frac{\tau^{\alpha-1}\left( \frac{4}{3} \mu + \mu_b \right)}{n_0} D_t^{\alpha}\Delta n' - \frac{\rho_0}{2(1-\bmass_0)^2} (\bmass'^2)_{tt} \\
				+\frac{\rho_0 \bmass_0}{c^2} \left( \bfu\cdot \bfu \right)_t  - \frac{\rho_0}{c^2} (\bfu \cdot \bfu)_{tt}  \end{multlined}
		\end{aligned}
	\end{equation}
	with $\tilde{b}=\frac{\gamma (1-\bmass_0)^2}{\rho_0} \left( \frac{1}{c_v} - \frac{1}{c_p} \right)$ and $\tilde{\beta}=(1-\bmass_0)^2 \frac{B}{2A}- \frac{(1-\bmass_0)\bmass_0}{2}$, where we have linked the fractional derivatives with $\tau = \tau_1 = \tau_2$ by setting $\alpha = \alpha_1=\alpha_2-1$ with $0 < \alpha \leq 1$ since $\alpha_1=1$, $\alpha_2=2$ leads to \eqref{eq: merged3}. 
	
	We can analogously reduce equation \eqref{eq: merged3 frac} for different cases, as done previously. Here, we simply state the Westervelt equation with time-fractional attenuation that is linear in $v'$:
	\begin{align} \label{eq: merged4 frac}
		\frac{1}{c^2 }p'_{tt} - \Delta p' -  \frac{\tau^{\alpha-1} b}{c^2 } D_t^{\alpha} \Delta p' - \frac{\beta}{\rho_0 c^4} (p'^2)_{tt}=   \rho_0 n_0 v'_{tt}
	\end{align}
	with $0 < \alpha \leq 1$, and $b=\frac{\gamma}{\rho_0} \left( \frac{1}{c_v} - \frac{1}{c_p} \right)$ and $\beta=\frac{B}{2A}+1$ as before.
	
	\subsubsection{Analysis of the system} 
	
	In the following, we consider the Westervelt--Rayleigh--Plesset system with time-fractional acoustic attenuation given by
	\begin{equation} \label{ibvp_fractionalWestRP}
		\left \{ \begin{aligned}
			&((1+ 2k(x)p) \pt)_t - c^2 \Delta p -b \tau^{\alpha-1}\Delta \Dtalpha p= \xi (R^3)_{tt} \quad \text{in } \Omega \times (0,T), \\[1mm]
			&(p, p_t)_{ \vert t=0} = (p_0, p_1), \quad p_{\vert \partial \Omega} =0,\\[1mm]
			& 		\rho_0 \left [ R \Rtt+ \frac32 \Rt^2 \right] \begin{multlined}[t]=h_0(R, \Rt)-p   \qquad \text{in } \Omega \times (0,T), \end{multlined} \\[1mm]
			&  (R, \Rt)_{ \vert t=0} = (R_0, R_1),
		\end{aligned} \right.
	\end{equation}
	with $\xi \in \R$ (corresponding to $\frac43 \pi c^2 \eta$ in the derivation) and $\alpha \in (0,1)$. Next, we wish to adapt the arguments for the Westervelt equation with strong acoustic damping to analyze \eqref{ibvp_fractionalWestRP}. Without loss of generality, we set the relaxation parameter to $\tau=1$. We recall first a well-posedness result from~\cite{ kaltenbacher2022inverse} on the Westervelt equation in non-bubbly media with time-fractional damping. As the one cannot exploit much of acoustic dissipation in this setting, more smoothness is needed from the data and the variable coefficient $k$ compared to Proposition~\ref{Prop:WellP_West}. 
	
	\subsection*{Notation} We use the notation $\Honethree =\,	\{ p \in H^3(\Om): p \vert_{\partial \Omega} = \Delta p \vert_{\partial \Omega} =0\}$. 	 
	
	\begin{proposition} [see Theorem 3.1 in~\cite{kaltenbacher2022inverse}]\label{Prop:WellP_fWest}
		Assume that $\Omega \subset \R^d$, $d \in \{1, 2,3\}$ is a bounded and $C^{2,1}$-regular domain. 	Let $c$, $b > 0$, and $\alpha \in (0,1)$. Furthermore, let $k \in W^{1, \infty}(\Om) \cap W^{2,4}(\Omega)$, $(p_0, p_1) \in   \Honethree \times \Honetwo, \text{ and } \ f \in L^2(0,T; \Honetwo).$
		There exists data size $\tdelta=\tdelta(T)>0$,  such that if
		\begin{equation}\label{datasmallness_fWestervelt}
			\|p_0\|_{\Hthree}+\|p_1\|_{\Htwo} + \|f\|_{L^2(\Htwo)} \leq \tdelta,
		\end{equation}
		then there is a unique solution of 
		\begin{equation} \label{IBVP_West}
			\left \{ \begin{aligned}
				&((1+ 2k(x)p) \pt)_t - c^2 \Delta p - b \Delta \Dtalpha p = f(x,t) \quad \text{in } \Omega \times (0,T), \\
				&p_{\vert \partial \Omega} =0, \quad  (p, p_t)_{ \vert t=0} = (p_0, p_1)
			\end{aligned} \right.
		\end{equation}
		in
		\begin{equation} \label{def_frakXp}
			\begin{aligned}
				\mathfrak{X}_p=  L^\infty(0,T; \Honethree) \cap W^{1, \infty}(0,T; \Honetwo) \cap H^2(0,T; \Honezero)
			\end{aligned}
		\end{equation}
		such that $1+ 2k p \geq \gamma >0 \ \text{ in }\ \Omega \times (0,T)$ for some $\gamma>0$.
		Furthermore, the solution satisfies the following bound:
		\begin{equation}
			\begin{aligned}
				\|p\|^2_{\frakXp} \lesssim	C(T) \left(\|p_0\|^2_{\Hthree}+\|p_1\|^2_{\Htwo} + \|f\|^2_{L^2(\Htwo)} \right).
			\end{aligned}
		\end{equation}
	\end{proposition}
	Under the assumptions of Proposition~\ref{Prop:WellP_fWest}, we can similarly to before introduce the mapping $\frakS: L^2(0,T; \Honetwo) \rightarrow \frakXp$, such that  $\frakS(f) = p$. Here, however Lipschitz continuity can only be established in a norm lower than that corresponding to the solution space. Let $\fone$, $\ftwo \in  L^2(0,T; \Honetwo)$ and denote $\pone = \frakS(\fone)$ and $\ptwo = \frakS(\ftwo)$. Then it can be shown analogously to the case with strong acoustic damping in \ref{Lemma: est diff p} that
	\begin{equation} \label{Lipschitz continuity fractional West}
		\begin{aligned}
			\| \pone-\ptwo\|_{C^1(\Ltwo)}+	\| \pone-\ptwo\|_{C(\Hone)} \lesssim \|\fone-\ftwo\|_{L^2(\Ltwo)}.
		\end{aligned}
	\end{equation}
	\indent To analyze the Westervelt--Rayleigh--Plesset system with time-fractional acoustic attenuation, we introduce the fixed-point mapping $	\frakT: \frakBR \ni \Rs \mapsto R$,	which maps $R^*$ taken from the set
	\begin{equation} \label{ball_fractional_R}
		\begin{aligned}
			\frakBR \negmedspace= \negmedspace\bigl\{ \negmedspace \Rs \negthickspace\in \negmedspace C^{2}(0,T; \Honetwo)\negmedspace: &\|\Rs_{tt}\|_{C(\Htwo)} \leq M, \\[1mm]
			&\|\Rs\|_{C(\Htwo)}+ \|\Rst\|_{C(\Htwo)} \leq m, \\[1mm]
			&\|\Rs \negthickspace- \negthinspace R_0\|_{C([0,T]; \Linf)} \negthinspace \leq \negmedspace \varepsilon_0, (\Rs \negthinspace, \Rst)_{t=0}\negmedspace = \negmedspace(R_0, R_1)  \bigr \}
		\end{aligned}
	\end{equation}
	to the solution of
	\begin{equation} \label{linearized_fractionalsystem}
		\left \{ \begin{aligned}
			& \Rtt = 	h(\Rs, \Rst, \frakS(f(\Rs))), \\
			&(R, \Rt) \vert_{t=0} = (R_0, R_1).
		\end{aligned} \right.
	\end{equation}
	Compared to the setting with strong damping, the main difference is the spatial regularity $\Htwo$ in \eqref{ball_fractional_R}, which originates from the need to have the acoustic right-hand side in $
	L^2(0,T; \Honetwo)$ as opposed to $\LtwoTLtwo$. We assume that $h$ is given by \eqref{def h}, where $\hzero$ satisfies the following assumptions. For any $\Rs \in \frakBR$,  
	\begin{equation}\label{assumption1 h_0 frak}
		\|\hzero(\Rs, \Rst)\|_{C(\Htwo)} \leq C(m, \varepsilon_0). 
	\end{equation}
	Furthermore, for any $\Rsone$, $\Rstwo \in \frakBR$,
	\begin{equation} \label{assumption2 h_0 frak}
		\begin{aligned}
			\|\hzero(\Rsone\negthickspace, \Rsone_t)\negthinspace - \negthinspace\hzero(\Rstwo \negthickspace, \Rstwo_t)\|_{C(\Hone)} \negthickspace \lesssim \|\Rsone\negthickspace-\Rstwo\|_{C^1(\Hone)}.
		\end{aligned}
	\end{equation} 
	We then have the following analogous result to that of Theorem~\ref{Thm:LocalWellp_WestRP}.
	\begin{theorem} \label{Thm:LocalWellp_fractionalWestRP}
		Assume that $\Omega \subset \R^d$, $d \in \{2,3\}$ is a bounded and $C^{2,1}$-regular domain. Let $c$, $b>0$, $\rho_0>0$, $\xi \in \R$, $\alpha \in (0,1)$, and $k\in W^{1, \infty}(\Om) \cap W^{2,4}(\Om)$, and let
		\begin{equation}
			(p_0, p_1) \in \Honethree \times \Honetwo, \quad (R_0, R_1) \in \Honetwo \times \Honetwo.
		\end{equation}
		Furthermore, let $\hzero$ satisfy \eqref{assumption1 h_0 frak} and \eqref{assumption2 h_0 frak}. Then there exist pressure data size $\tdelta_p>0$, bubble data size $\tdelta_R$, and time $\tilde{T}$, such that if
		\begin{equation}
			\|p_0\|_{\Hthree}+\|p_1\|_{\Htwo} \leq \tdelta_p, \quad 	\|R_0\|_{\Htwo}+\|R_1\|_{\Htwo} \leq \tdelta_R, \quad \text{and } \ T \leq \tilde{T},
		\end{equation}
		then there is a unique $(p, R) \in \frakXp \times \frakBR$ that solves \eqref{ibvp_fractionalWestRP_general} with $\calA p = - b \Delta \Dt^\alpha p$, where the space $\frakXp$ is defined in \eqref{def_frakXp}.
	\end{theorem}
	\begin{proof}
		The proof follows analogously to the proof of Theorem~\ref{Thm:LocalWellp_WestRP}. For $\Rs \in \frakBR$, the smallness condition in \eqref{datasmallness_fWestervelt} is now replaced by 
		\begin{equation} \label{smallness_condition_fractionalproof}
			\begin{aligned}
				\|p_0\|_{\Hthree}+\|p_1\|_{\Htwo} +\xi \| ((\Rs)^3)_{tt}\|_{L^2(\Htwo)} \leq \tdelta
			\end{aligned}
		\end{equation}
		with $\tdelta$ set by the smallness condition in \eqref{datasmallness_fWestervelt}. Since
		\begin{equation}
			\begin{aligned}
				&\| ((\Rs)^3)_{tt}\|_{L^2(\Htwo)} \\
				%	=&\, \| 3 (\Rs)^2 \Rstt+ 6 \Rs (\Rst)^2\|_{L^2(\Htwo)} \\
				\leq&\, 3\|\Rs\|^2_{L^\infty(\Htwo)}\|\Rstt\|_{L^2(\Htwo)}+ 6 \|\Rs\|_{L^\infty(\Htwo)}\|\Rst\|_{L^\infty(\Htwo)}\|\Rst\|_{L^2(\Htwo)}  \\
				\leq&\, 3m^2 \sqrt{T} M+ 6m^3 \sqrt{T},
			\end{aligned}
		\end{equation}
		condition \eqref{smallness_condition_fractionalproof} is fulfilled for small enough $\tdeltap$ and $T$, similarly to before. The rest of the arguments for proving that $\frakT$ is a self-mapping follow in a same manner. When proving strict contractivity of $\frakT$, here we can only exploit the Lipschitz continuity of the pressure field $p$ in the sense of \eqref{Lipschitz continuity fractional West}. Thus, we can use the following bound:
		\begin{equation} \label{contractivity_est fract}
			\begin{aligned}
				\|\tR\|_{C^2(\Hone)} \negthickspace
				\lesssim \negthinspace\| h(\Rsone\negthinspace, \Rsone_t\negthinspace, \calS(f(\Rsone)))\negthinspace-h(\Rstwo\negthinspace, \Rstwo_t\negthinspace \calS(f(\Rstwo)))\|_{C(\Hone)}
			\end{aligned}
		\end{equation}
		together with \eqref{Lipschitz continuity fractional West} to prove that $\frakT$ is strictly contractive in the norm in \\
		\noindent $C^2([0,T]; \Honezero)$, as opposed to $C^2([0,T]; \Honetwo)$. The rest of the arguments follow analogously to Theorem~\ref{Thm:LocalWellp_WestRP} and we thus omit them here.
	\end{proof}
	
\end{appendices}

	%%%%%%%%%%%%%%%%%%%%%%%%%%%%%%%
	\bibliographystyle{siamplain}
	\bibliography{references}
\end{document}